\documentclass{amsart}

\usepackage{amsthm}
\usepackage{amsfonts}

\newcommand{\C}{{\mathbb{C}}}

\newtheorem{theorem}{Theorem}[section]

\newtheorem{lemma}[theorem]{Lemma}
\newtheorem{proposition}[theorem]{Proposition}

\newtheorem{definition}[theorem]{Definition}

\newtheorem{question}[theorem]{Question}

\title{Permutation actions on equivariant cohomology}
\author{Julianna S. Tymoczko}
\address{Department of Mathematics, University of Iowa, 14 MacLean Hall, Iowa City, IA 52242-1419}
\email{tymoczko@umich.edu}
\thanks{The author was partially supported by NSF grant 0402874.}


\begin{document}

\begin{abstract}
This survey paper describes two geometric representations of the permutation group using the tools of toric topology.  These actions are extremely useful for computational problems in Schubert calculus.  The (torus) equivariant cohomology of the flag variety is constructed using the combinatorial description of Goresky-Kottwitz-MacPherson, discussed in detail.  Two permutation representations on equivariant and ordinary cohomology are identified in terms of irreducible representations of the permutation group.  We show how to use the permutation actions to construct divided difference operators and to give formulas for some localizations of certain equivariant classes.   

This paper includes several new results, in particular a new proof of the Chevalley-Monk formula and a proof that one of the natural permutation representations on the equivariant cohomology of the flag variety is the regular representation.  Many examples, exercises, and open questions are provided.
\end{abstract}

\maketitle

\section{Introduction}
This paper constructs two permutation representations on the cohomology (ordinary and equivariant) of a particular algebraic variety called the flag variety.  There are three main reasons to consider these permutation representations.  First, they are examples of geometric representations and so have intrinsic interest to both geometers and representation theorists.  Second, they are computationally useful: they give effective tools for calculating in the (equivariant and ordinary) cohomology ring of the flag variety, which is the essential goal of Schubert calculus.  Third, they provide combinatorial and algebraic tools for answering geometric questions and geometric tools for answering combinatorial and algebraic questions, because deep properties of matrix algebra and the combinatorics of the permutation group are embedded in the geometry of the flag variety.

We identify each permutation action explicitly, giving the new result that one is the regular representation, and the previously known result that the other is (several copies of) the trivial representation.  To demonstrate the computational effectiveness of these actions, we give two new localization formulas for certain equivariant classes called Schubert classes as well as a new proof of the Chevalley-Monk formula, which describes the product of particular Schubert classes in the ordinary cohomology of the flag variety.  

Though it includes several new results, this paper is primarily expository and is written with a broad audience in mind.  We assume no particular algebraic geometric, algebraic topological, or combinatorial knowledge.  Definitions of the basic objects, including the flag variety, follow in later sections.  The rest of this introduction provides an overview of the motivation behind these fields of study as well as the goals of the paper.

One question this paper focuses on is a central problem in geometric representation theory.  The basic object of classical representation theory is a representation, namely a complex vector space $V$ with a linear action of a finite group $G$.  Classical representation theory asks: 
\begin{itemize}
\item Find the `minimal' representations of $G$, namely representations $V$ which have no subspace that carries a $G$-action other than $\{0\}$.  (These representations are called {\em irreducible}.)
\item Given a representation $V$ of $G$, decompose $V$ into a direct sum of irreducible representations.
\end{itemize} 
Geometric representation theorists build representations using either a geometric group action on a variety or, more commonly, a geometrically-natural group action on the cohomology of a variety.  When a geometric representation has been established, it provides a dictionary between algebraic properties of the representation (e.g., its dimension) and geometric properties of the underlying variety (e.g., the number of its connected components).  This is the essence of the Langlands program, an intense area of research in number theory that establishes correspondences between automorphic forms and the geometry of an infinite-dimensional analogue of the flag variety.  
The permutation representations described in this paper are an example of geometric representations, and a rare case of geometric representations that can be described very explicitly.

The cohomology of the flag variety has a natural basis of Schubert classes $p_w$ indexed by permutations $w \in S_n$.  This paper also addresses the central---and open---question of Schubert calculus: what are the coefficients $c_{uv}^w$ in the expression $p_up_v = \sum c_{uv}^w p_w$?   Schubert calculus originated in the nineteenth century to calculate the intersections of various linear subspaces, which turn out to enumerate the $c_{uv}^w$ in ordinary cohomology.  Miraculously, in the analogous question for Grassmannians, the coefficients $c_{uv}^w$ give the tensor product multiplicities of certain irreducible representations.  This means that algebraic questions about the cohomology of the flag variety are tied to deep geometric and representation theoretic questions.  The permutation representations discussed here help to calculate efficiently in the (ordinary and equivariant) cohomology ring of the flag variety.  We give several examples, for instance using the representations to construct divided difference operators, a standard computational tool in Schubert calculus.   

There are several reasons to focus on the flag variety.  First, it is both an essential and a bridge object in geometry, algebra, and combinatorics.  For instance, the flag variety generalizes to $n$ dimensions the tangent space to a curve at point.  At the same time, all the combinatorial data of the permutation group is embedded in the geometry of the flag variety.  Second, few algebraic varieties can be described explicitly and entirely; flag varieties are a rare example.  Third, the flag variety is in many cases easier to work with than other varieties to which it is closely related, including the Grassmannians of $k$-dimensional planes in $\C^n$.  This means permutation representations on the cohomology of the flag variety give particularly rich combinatorial and algebraic data.

The main technique we use is the Goresky-Kottwitz-MacPherson (GKM) approach to equivariant cohomology.  GKM theory, as it is often called, gives a purely combinatorial algorithm for constructing the equivariant cohomology of certain algebraic varieties.  (GKM theory describes conditions on a variety under which the localizations of its equivariant cohomology can be characterized completely by combinatorial graphs.)  In the case of the flag variety, the equivariant cohomology ring is built directly from an important combinatorial object called the Bruhat graph of the permutation group.  The ordinary cohomology can be recovered from the equivariant cohomology ring.  One goal of this paper is to present the material in a way that combinatorists will find accessible, as well as topologists.

This paper is written for the flag variety $GL_n(\C)/B$ to reach the widest possible audience.  Nonetheless, the methods generalize immediately to the flag variety $G/B$ of an arbitrary linear algebraic group.  Many examples are given; comments following the main results describe how they apply to more general flag varieties.  A secondary goal of this paper is to stimulate research to extend work in type $A_n$ to other Lie types, particularly by geometers and topologists.  Similar projects, for instance in the study of singularities of Schubert varieties, have been extremely productive areas recently in geometry and combinatorics (\cite{BL} has an overview).

The author gratefully acknowledges the organizers of the Osaka conference in toric topology, where the original version of this paper was given.

\section{GKM theory to compute equivariant cohomology}

One way to compute equivariant cohomology of a projective algebraic variety $X$ with the action of a torus $T$ is to use the localization map $H^*_T(X) \rightarrow H^*_T(X^T)$ induced from inclusion of the fixed points $X^T$ into $X$.  Under suitable circumstances, the localization map is an injection.  Under even more suitable circumstances, combinatorial conditions on a tuple of localizations determine whether the tuple is in the image of the localization map.  These conditions are often known as GKM theory, after work of Goresky-Kottwitz-MacPherson that explicitly described an elegant combinatorial framework under which the localization map is injective.

This section contains a brief sketch of GKM theory.  For more detail, the reader is referred to the original paper \cite{GKM}, the beautiful description of A.\ Knutson and T.\ Tao \cite{KT}, or the expository article \cite{T1}.

Let $X$ be a complex projective algebraic variety with an (algebraic) action of a complex torus $T = \C^* \times \cdots \times \C^*$.  If $X$ has only isolated $T$-fixed points, then the closure of each one-dimensional $T$-orbit $O$ is in fact homeomorphic to $\C\mathbb{P}^1$.  The boundary $\overline{O} \backslash O$ consists of exactly two $T$-fixed points, often called the north and south poles and denoted $N_O$ and $S_O$.  If the weight of the torus action on the tangent space at the north pole $\mathcal{T}_{N_O}(\overline{O})$ is $\alpha$, then the $T$-weight on $\mathcal{T}_{S_O}(\overline{O})$ is $-\alpha$.  In particular, the $T$-weight on the fixed points in the closure $\overline{O}$ is determined up to sign; through a slight abuse of notation, we call it the $T$-weight on the orbit $O$.

Under the following three conditions, the collection of $T$-fixed points and one-dimensional $T$-orbits in $X$ gives a so-called ``balloon sculpture" that encodes all the data of $H^*_T(X)$.  In particular, assume that:
\begin{enumerate}
\item $X$ has finitely many $T$-fixed points;
\item $X$ has finitely many one-dimensional $T$-orbits;
\item and $X$ is {\em equivariantly formal}.
\end{enumerate}
For us, `equivariantly formal' means a particular spectral sequence degenerates.  It is implied by any of the following conditions, as well as many others \cite[Section 1.2 and Theorem 14.1]{GKM}: $X$ is a smooth complex projective variety, $X$ has no odd-dimensional cohomology, or $X$ has a $T$-stable complex CW-decomposition.  

Conditions (1)--(3) are often called the GKM conditions.  A variety $X$ that satisfies all three is known as a GKM space.  The flag variety, Grassmannians, and Schubert varieties are GKM spaces.  Note that a GKM space may be singular.

\begin{theorem} (Goresky-Kottwitz-MacPherson)
If $X$ is a GKM space then 
\[H^*_T(X) \cong \left\{ (p_v) \in \left(H^*_T(pt)\right)^{|X^T|} : \begin{array}{c} \textup{ for each one-dimensional $T$-orbit $O$ with} \\ \textup{ poles $N_O$ and $S_O$ and $T$-weight $\alpha$, } \\ p_N - p_S \in \langle \alpha \rangle \end{array}\right\}.\]
\end{theorem}

Goresky-Kottwitz-MacPherson's contribution was to identify a large family of varieties for which the image of the map could be characterized concretely and combinatorially.  That $H^*_T(X)$ injects into the tuples of localizations was proven by F.\ Kirwan and anticipated by earlier results of Borel, Atiyah, Hsiang, and Quillen (see \cite[Section 1.7]{GKM} for a thorough history).  Chang and Skjelbred identified the image of this map in a more general setting in \cite[Lemma 2.3]{CS}.  The GKM results were later extended in many ways, including to varieties with non-isolated fixed points in \cite{GH} and with infinite number of one-dimensional orbits in \cite{BCS}.

We distill the geometric data of one-dimensional orbits and fixed points into a purely combinatorial graph called the moment graph of $X$.

\begin{definition}  Let $X$ be a GKM space and choose a one-dimensional subtorus $T'$ of $T$ that is generic in the sense that it fixes no one-dimensional orbit of $T$.  

The moment graph of $X$ is a labeled, directed combinatorial graph with 
\begin{enumerate}
\item vertex set $X^T$, 
\item an edge between $v,w \in X^T$ if and only if there is a one-dimensional $T$-orbit $O$ such that $\{v,w\} \subseteq \overline{O}$,
\item the label $\alpha$ on the edge between $v,w$ if and only if the $T$-weight on $\mathcal{T}_N(\overline{O})$ is $\pm \alpha$, and
\item the edge between $v$ and $w$ is directed consistent with the flow of $T'$ on $O$.
\end{enumerate}
\end{definition}

If $X$ is a Hamiltonian $T$-space with a moment map, then the zero- and one-dimensional skeleton of the moment map image coincides with the moment graph.  The moment graph retains data about the edge-labels and directions that come from the ambient Euclidean space of the moment map image.

The moment graphs for many varieties of combinatorial or algebraic interest, such as flag varieties, Schubert varieties, and Grassmannians, have been studied independently by combinatorists and are closely related to Bruhat order.  The moment graphs of flag varieties are described in more detail in the next section.

The ordinary cohomology can be recovered from the equivariant cohomology.  In fact, there is a ring isomorphism
\[H^*(X) \cong \frac{H^*_T(X)}{\mathfrak{m} H^*_T(X)}\]
where $\mathfrak{m}$ is the maximal ideal in $H^*_T(pt)$.  The ring $H^*_T(pt)$ is the symmetric algebra in the cotangent to the torus, though we often use instead the polynomial ring $\C[t_1, \ldots, t_n]$ in the tangent of the torus.

We conclude by restating the GKM theorem in this combinatorial context.  We denote a directed edge from $v$ to $w$ by $v \mapsto w$.

\begin{theorem} (Goresky-Kottwitz-MacPherson, version 2)
If $X$ is GKM then
\[H^*_T(X) \cong \left\{ (p_v) \in \left(H^*_T(pt)\right)^{|X^T|} : \begin{array}{c} \textup{ for each edge $v \mapsto w$ labeled $\alpha$,} \\ p_v - p_w \in \langle \alpha \rangle \end{array}\right\}.\]
\end{theorem}

\subsection{GKM theory for flag varieties}

In the most general setting for flag varieties, $G$ is a complex reductive linear algebraic group and $B$ is a Borel subgroup containing a maximal torus $T$.  The flag variety is the quotient $G/B$.  The torus $T$ acts on $G/B$ by multiplication: if $t \in T$ and $[g] \in G/B$ then $t \cdot [g] = [tg]$.  The flag variety $G/B$ is a GKM space with this torus action.  The rest of this section describes the moment graph for flag varieties.  In the example that we use most, $G$ is the group of $n \times n$ invertible matrices with complex coefficients and $B$ is the subgroup of invertible upper-triangular matrices.  (Many mathematicians who study flag varieties consider this example exclusively.)  

The flag variety $GL_n(\C)/B$ has a natural geometric description.  Let $\C^n$ be an $n$-dimensional complex vector space with a fixed basis.  A flag is a collection of nested vector subspaces $V_1 \subseteq V_2 \subseteq \cdots \subseteq V_n = \C^n$, where each $V_i$ is an $i$-dimensional complex vector space.  The coset $[g] \in GL_n(\C)/B$ represents the flag $V_1 \subseteq V_2 \subseteq \cdots \subseteq \C^n$ if the first $i$ columns of the matrix $g$ span $V_i$ for each $i$.  There is more than one matrix representative for the same flag; the reader may observe that the matrices $gB$ are exactly those that represent the flag $[g]$.

For $GL_n(\C)/B$, the torus $T$ consists of the $n \times n$ invertible diagonal matrices.  The action $t \cdot [g] = [tg]$ is equivalent to the geometric action
\[t \cdot \left( V_1 \subseteq V_2 \subseteq \cdots \subseteq \C^n \right) = \left( tV_1 \subseteq tV_2 \subseteq \cdots \subseteq t \C^n\right) .\]
To identify the $T$-fixed flags, note that $T$ must preserve each subspace $V_i$ in the flag.  This implies that $V_1$ is an eigenspace for $T$, that $V_2$ is the direct sum of two eigenspaces for $T$, and so on.  The eigenvectors for $T$ are precisely the basis vectors, so the $T$-fixed flags are exactly the flags whose subspaces are spanned by a permutation of the basis vectors.  In other words, the $T$-fixed flags are the flags $[w]$ for each $n \times n$ permutation matrix $w$.

We write $S_n$ both for the $n \times n$ permutation matrices and for the group of permutations on the set $\{1,2,\ldots,n\}$.  Furthermore, we will not distinguish between matrices and permutations in our notation.  Our convention is that if $e_i$ is the $i^{th}$ basis vector, then $we_i = e_{w(i)}$.

The $T$-stable curves in $GL_n/B$ are identified in the same way as the $T$-fixed flags.  Loosely speaking, they consist of permutation flags with an additional nonzero entry.  An excellent {\bf exercise} is to formulate combinatorial conditions on $(i,j)$ that characterize when a permutation matrix $w$ with a single extra nonzero entry in position $(i,j)$ does {\em not} span a single flag.  $T$-stable curves correspond to flags with each $V_i$ spanned by coordinate axes except one, which is in the span of two coordinate axes.  Figure \ref{curve examples} shows matrices whose closure in the flag variety is a $T$-stable curve.  (In each case the closure consists of one additional $T$-fixed flag.)
\begin{figure}[h]	
\[\left( \begin{array}{lll} a & 1 & 0 \\ 1 & 0 & 0 \\ 0 & 0 & 1 \end{array} \right) \hspace{.5in}
	\left( \begin{array}{lll} b & 1 & 0 \\ 0 & 0 & 1 \\ 1 & 0 & 0 \end{array} \right) \hspace{.5in}
	\left( \begin{array}{lll} 0 & 1 & 0 \\ c & 0 & 1 \\ 1 & 0 & 0 \end{array} \right)
	\]
\caption{Dense sets in $T$-stable curves, with parameters $a, b, c \in \C$} \label{curve examples}
\end{figure}

The following result, presented originally by J.\ Carrell in \cite{C}, summarizes the data involved in the moment graph of the flag variety. Recall that the {\em length} of a permutation $w$ is the minimum number of simple transpositions $s_i = (i,i+1)$ required in a factorization $w = s_{i_1} s_{i_2} \cdots s_{i_k}$.

\begin{proposition}  \label{moment graph of flag variety}The moment graph of the flag variety is a combinatorial graph with the following properties.
\begin{enumerate}
\item Its vertices are permutations $[w] \in S_n$.
\item There is an edge between each pair of permutations $[w] \leftrightarrow [(ij)w]$.
\item The label on the edge $[w] \leftrightarrow [(ij)w]$ is $t_i - t_j$.
\item The edge is directed $[w] \mapsto [(ij)w]$ if and only if $\ell(w) > \ell((ij)w)$.
\end{enumerate}
\end{proposition}

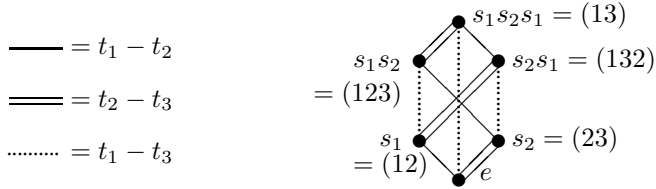
\begin{figure}[h]
\begin{picture}(360,60)(0,-30)

\put(200,-30){\circle*{5}}
\put(185,-15){\circle*{5}}
\put(215,-15){\circle*{5}}
\put(185,15){\circle*{5}}
\put(215,15){\circle*{5}}
\put(200,30){\circle*{5}}

\put(200,-30){\line(-1,1){15}}
\put(215,-15){\line(-1,1){30}}
\put(215,15){\line(-1,1){15}}

\put(201,-31){\line(1,1){15}}
\put(199,-29){\line(1,1){15}}
\put(186,-16){\line(1,1){30}}
\put(184,-14){\line(1,1){30}}
\put(186,14){\line(1,1){15}}
\put(184,16){\line(1,1){15}}

\multiput(185,-15)(0,2){15}{\circle*{1}}
\multiput(200,-30)(0,2){30}{\circle*{1}}
\multiput(215,-15)(0,2){15}{\circle*{1}}

\put(30,20){\line(1,0){20}}
\put(30,1){\line(1,0){20}}
\put(30,-1){\line(1,0){20}}
\multiput(30,-20)(2,0){10}{\circle*{1}}

\put(53,18){$=t_1-t_2$}
\put(53,-2){$=t_2-t_3$}
\put(53,-22){$=t_1-t_3$}

\put(220,-17){$s_2 = (23)$}
\put(220,13){$s_2s_1 = (132)$}
\put(170,-17){$s_1$}
\put(160,-27){$ = (12)$}
\put(160,13){$s_1s_2$}
\put(145,0){$ = (123)$}
\put(205,30){$s_1s_2s_1 = (13)$}
\put(208,-30){$e$}
\end{picture}
\caption{The moment graph for $GL_3(\C)/B$} \label{A2 flags}
\end{figure}

The GKM theorem now determines the classes in $H^*_T(GL_n(\C)/B)$: equivariant classes are tuples of polynomials so that the difference between polynomials joined by an edge is a multiple of the label on that edge.  Checking edge-by-edge, the reader can verify which examples in Figure \ref{examples of classes} are classes.
\begin{figure}[h]
\begin{picture}(360,60)(90,-30)
\put(200,-30){\circle*{5}}
\put(185,-15){\circle*{5}}
\put(215,-15){\circle*{5}}
\put(185,15){\circle*{5}}
\put(215,15){\circle*{5}}
\put(200,30){\circle*{5}}

\put(200,-30){\line(-1,1){15}}
\put(215,-15){\line(-1,1){30}}
\put(215,15){\line(-1,1){15}}

\put(201,-31){\line(1,1){15}}
\put(199,-29){\line(1,1){15}}
\put(186,-16){\line(1,1){30}}
\put(184,-14){\line(1,1){30}}
\put(186,14){\line(1,1){15}}
\put(184,16){\line(1,1){15}}

\multiput(185,-15)(0,2){15}{\circle*{1}}
\multiput(200,-30)(0,2){30}{\circle*{1}}
\multiput(215,-15)(0,2){15}{\circle*{1}}

\put(220,-17){$0$}
\put(220,13){$t_1-t_3$}
\put(150,-17){$t_1-t_2$}
\put(150,13){$t_1-t_2$}
\put(205,30){$t_1-t_3$}
\put(210,-32){$0$}

\put(350,-30){\circle*{5}}
\put(335,-15){\circle*{5}}
\put(365,-15){\circle*{5}}
\put(335,15){\circle*{5}}
\put(365,15){\circle*{5}}
\put(350,30){\circle*{5}}

\put(350,-30){\line(-1,1){15}}
\put(365,-15){\line(-1,1){30}}
\put(365,15){\line(-1,1){15}}

\put(351,-31){\line(1,1){15}}
\put(349,-29){\line(1,1){15}}
\put(336,-16){\line(1,1){30}}
\put(334,-14){\line(1,1){30}}
\put(336,14){\line(1,1){15}}
\put(334,16){\line(1,1){15}}

\multiput(335,-15)(0,2){15}{\circle*{1}}
\multiput(350,-30)(0,2){30}{\circle*{1}}
\multiput(365,-15)(0,2){15}{\circle*{1}}

\put(370,-17){$0$}
\put(370,13){$t_1-t_3$}
\put(300,-17){$t_1-t_2$}
\put(300,13){$t_1-t_2$}
\put(355,30){$t_2-t_3$}
\put(360,-32){$0$}

\put(390,-5){NOT in}
\put(380,-20){$H^*_T(GL_n(\C)/B)$}
\end{picture}
\caption{One class in $H^*_T(GL_n(\C)/B)$ and another not} \label{examples of classes}
\end{figure}
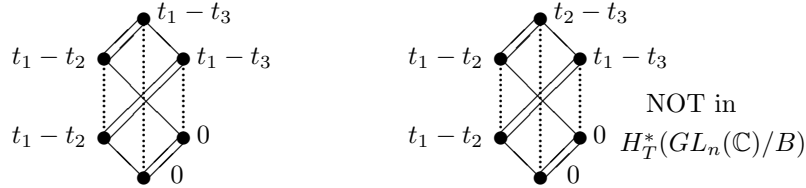

Having described the moment graph for $GL_n(\C)/B$, we remark that the case of the general flag variety $G/B$ is almost exactly the same.  It, too, has a geometric description in terms of nested subspaces (at least for classical Lie types), though we will not discuss this at length \cite[Lectures 16 and 18]{FH}.  The other flag varieties consist of nested linear subspaces $V_1 \subseteq \cdots V_N$ such that each $V_i$ is contained in its own orthogonal complement with respect to a particular linear form, namely $\langle v, w \rangle = 0$ for each $v,w$ in each $V_i$.  Given a standard basis $e_1, \ldots, e_N$, we recommend the following  choices:
\begin{itemize}
\item in type $B_n$, the symmetric form that is nonzero only on basis vectors $\langle e_i, e_{2n+2-i} \rangle = 1$ (for $N=2n+1$);
\item in type $C_n$, any alternating form that is nonzero only on basis vectors $\langle e_i, e_{2n+1-i} \rangle = \pm 1$ (for $N = 2n$);
\item and in type $D_n$, the symmetric form that is nonzero only on basis vectors $\langle e_i, e_{2n+1-i} \rangle = 1$ (for $N=2n$).
\end{itemize}
Though not entirely standard, these linear forms preserve many algebraic properties of $GL_n(\C)/B$.  For instance, with these choices each flag variety will naturally be written $G/B$ for the subgroup $G \subseteq GL_N(\C)$ preserved by the linear form and for a subgroup $B$ of upper-triangular matrices in $GL_N(\C)$.

The Weyl group $W$ can be defined as the quotient $N(T)/T$ of the normalizer of the torus in $G$.  In the case of $GL_n(\C)/B$, the Weyl group is the permutation group.  The general Weyl group is completely analogous.  The $T$-fixed points of $G/B$ are the Weyl flags $[w]$ for $w \in W$.  Just as the permutation group is generated by reflections $(ij)$, the general Weyl group is generated by reflections denoted $s_{\alpha}$.  The edges containing $[w]$ in the moment graph are precisely $[s_{\alpha}w] \leftrightarrow [w]$ for each reflection $s_{\alpha}$ and each Weyl flag $[w]$.  The edge $[s_{\alpha}w] \leftrightarrow [w]$ is labeled $\alpha$.  

Each Weyl group has simple reflections $s_i$ that are to general reflections $s_{\alpha}$ what the permutations $(i,i+1) \in S_n$ are to $(ij)$.  In particular, the simple reflections generate $W$ and each reflection $s_{\alpha}$ is of the form $s_{i_k} \cdots s_{i_1}s_{i_0}s_{i_1} \cdots s_{i_k}$ for some $i_0, \ldots, i_k$.  Figure \ref{table of simple reflections} gives a list of simple reflections in classical Lie types.  (The flag variety of type $A_n$ is $GL_{n+1}(\C)/B$.)
\begin{figure}[h]
$\begin{array}{|c|c|c|}
\cline{1-3} \textup{ Lie type } & \textup{ Simple reflections } & \\
\cline{1-3} A_n & s_i=(i,i+1) & i = 1, \ldots, n \\
\cline{1-3} B_n & s_i=(i,i+1)(2n+1-i,2n+2-i) & i=1, \ldots, n-1 \\
& s_n=(n,n+2) & \\
\cline{1-3} C_n & s_i=(i,i+1)(2n-i,2n+1-i) & i=1, \ldots, n-1 \\
& s_n=(n,n+1) & \\
\cline{1-3} D_n & s_i=(i,i+1)(2n-i,2n+1-i) & i = 1, \ldots, n-1 \\
& s_n=(n,n+2)(n-1,n+1) & \\
\hline \end{array}$
\caption{Simple reflections for $W$ of classical Lie type} \label{table of simple reflections}
\end{figure}
For the reader's convenience, Figure \ref{roots and action} gives the positive roots of each classical Lie type together with the action of the simple transpositions on the roots.  (The action of an arbitrary element $w \in W$ can be determined from the action of the simple transpositions.)  
\begin{figure}[h]
$\begin{array}{|c|c|c|}
\cline{1-3} \textup{ Lie type } & \textup{ Positive roots } & \textup{ Action of $s_i$ } \\
\cline{1-3} A_n & t_i - t_j \textup{ for } 1 \leq i < j \leq n & s_i \textup{ exchanges } t_i \leftrightarrow t_{i+1} \\
\cline{1-3} B_n & t_i - t_j \textup{ for } 1 \leq i < j \leq n & s_i \textup{ exchanges } t_i \leftrightarrow t_{i+1} \textup{ for } i = 1, \ldots, n-1 \\
& t_i + t_j  \textup{ for } 1 \leq i < j \leq n & s_n \textup{ negates } t_n \mapsto -t_n \\
& t_i \textup{ for } 1 \leq i \leq n & \\
\cline{1-3} C_n & t_i - t_j \textup{ for } 1 \leq i < j \leq n & s_i \textup{ exchanges } t_i \leftrightarrow t_{i+1} \textup{ for } i = 1, \ldots, n-1 \\
& t_i + t_j  \textup{ for } 1 \leq i < j \leq n & s_n \textup{ negates } t_n \mapsto -t_n \\
& 2t_i \textup{ for } 1 \leq i \leq n & \\
\cline{1-3} D_n & t_i - t_j \textup{ for } 1 \leq i < j \leq n & s_i \textup{ exchanges } t_i \leftrightarrow t_{i+1} \textup{ for } i = 1, \ldots, n-1 \\
& t_i + t_j  \textup{ for } 1 \leq i < j \leq n & s_n \textup{ exchanges and negates } t_n \leftrightarrow -t_{n-1} \\
\hline
\end{array}$
\caption{Roots and the reflection action for classical Lie types} \label{roots and action}
\end{figure}
Roots are often written in terms of the basis $\alpha_i$ of simple roots, namely those negated by the simple reflections.  The {\em length} of $w \in W$ is the minimal number of simple reflections required to write $w = s_{i_1} s_{i_2} \cdots s_{i_k}$.  The edge $[w] \mapsto [s_{\alpha}w]$ is directed so that $\ell(w) > \ell(s_{\alpha}w)$.  The reader is referred to \cite{BB} for more on the combinatorics of Weyl groups.  Figure \ref{C2 flags} shows the moment graph for the flag variety of type $C_2$.  (The moment graph for the flag variety of type $B_2$ is identical, except that the roles of $s_1$ and $s_2$ and their corresponding simple roots are everywhere exchanged.)

\begin{figure}[h]
\begin{picture}(360,70)(0,-35)
\put(200,-30){\circle*{5}}
\put(220,-30){\circle*{5}}
\put(180,-10){\circle*{5}}
\put(240,-10){\circle*{5}}
\put(180,10){\circle*{5}}
\put(240,10){\circle*{5}}
\put(200,30){\circle*{5}}
\put(220,30){\circle*{5}}

\put(199,-30){\line(-1,1){20}}
\put(219,-30){\line(-1,1){40}}
\put(239,-10){\line(-1,1){40}}
\put(239,10){\line(-1,1){20}}
\put(201,-30){\line(-1,1){20}}
\put(221,-30){\line(-1,1){40}}
\put(241,-10){\line(-1,1){40}}
\put(241,10){\line(-1,1){20}}

\put(200,-30){\line(1,0){20}}
\put(180,-10){\line(1,0){60}}
\put(180,10){\line(1,0){60}}
\put(200,30){\line(1,0){20}}

\multiput(220,-30)(2,2){10}{\circle*{1}}
\multiput(200,-30)(2,2){20}{\circle*{1}}
\multiput(180,-10)(2,2){20}{\circle*{1}}
\multiput(180,10)(2,2){10}{\circle*{1}}

\multiput(180,-10)(0,5){4}{\line(0,1){2}}
\multiput(200,-30)(0,5){12}{\line(0,1){2}}
\multiput(220,-30)(0,5){12}{\line(0,1){2}}
\multiput(240,-10)(0,5){4}{\line(0,1){2}}

\put(20,30){\line(1,0){20}}
\put(20,11){\line(1,0){20}}
\put(20,9){\line(1,0){20}}
\multiput(20,-10)(5,0){4}{\line(1,0){2}}
\multiput(20,-30)(2,0){10}{\circle*{1}}

\put(43,28){$= t_1-t_2 = \alpha_1$}
\put(43,8){$= 2t_2 = \alpha_2$}
\put(43,-12){$= t_1+t_2 = \alpha_1+\alpha_2$}
\put(43,-32){$= 2t_1 = 2\alpha_1+\alpha_2$}

\put(190,-32){$e$}
\put(225,-32){$s_1$}
\put(245,-12){$s_1s_2$}
\put(165,-12){$s_2$}
\put(245,8){$s_1s_2s_1$}
\put(158,8){$s_2s_1$}
\put(168,28){$s_2s_1s_2$}
\put(225,28){$s_2s_1s_2s_1$}

\put(300,0){\vector(1,2){10}}
\put(290,-10){up}
\end{picture}
\caption{The moment graph for the flag variety of type $C_2$} \label{C2 flags}
\end{figure}
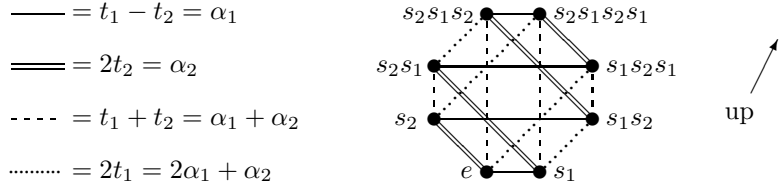

\subsection{Canonical classes}

This section discusses a basis for the equivariant cohomology of flag varieties that arises from the GKM construction and from the underlying geometry of the flag variety.

The cohomology of flag varieties has a natural geometric basis of Schubert classes.  A Schubert cell is the collection of flags $[Bw]$ for a particular $w \in S_n$.  The Schubert cells partition the flag variety (indeed, the double cosets $BwB$ partition the group $G$).  Moreover, this decomposition into Schubert cells forms a CW-complex.  The flag variety is smooth so Poincar\'{e} duality gives cohomology classes dual to the classes of the closures $\overline{[Bw]}$ in $H_*(G/B)$.  These cohomology classes are called Schubert classes and denoted $\Omega_w$.  The central question of Schubert calculus is to identify the structure constants $c_{uv}^w$ defined by 
\[\Omega_u \Omega_v = \sum_{w \in S_n} c_{uv}^w \Omega_w.\]
This question is almost completely open in the case of the flag variety $GL_n(\C)/B$ and is almost entirely unattacked for general $G/B$.  The classical results of Schubert calculus address the Grassmannian $G(k,n)$ of $k$-planes in $\C^n$, which is a partial flag variety.  Grassmannians have Schubert classes $p_{\lambda} \in H^*(G(k,n))$ indexed by partitions $\lambda$.  The structure constants $c_{\lambda \nu}^{\mu}$ defined by
\[p_{\lambda} p_{\nu} = \sum_{\mu} c_{\lambda \nu}^{\mu} p_{\mu}\]
turn out to count geometric intersection numbers, for instance the number of lines which intersect four generic lines in $\mathbb{P}^3$.  Moreover, if $V_{\lambda}$ denotes the irreducible $S_n$-representation corresponding to the cycle type $\lambda$, then the structure constants are also the tensor product multiplicities
\[V_{\lambda} \otimes V_{\nu} = \sum_{\mu}  c_{\lambda \nu}^{\mu} V_{\mu}.\]
\cite{F} is a classic reference for both combinatorial and geometric Schubert calculus.

The GKM construction leads to a natural combinatorial basis for the cohomology.  The equivariant classes of this basis are called {\em canonical classes}.  For a GKM space which is well-behaved (more on that in a moment), the canonical class $p_v$ corresponding to a fixed point $v$ is constructed by
\begin{enumerate}
\item \label{zero below} setting $p_v(u)=0$ if there is no directed chain $u \mapsto u_1 \mapsto \cdots \mapsto v$ in the moment graph;
\item \label{product on bottom vertex} setting $p_v(v) = \prod_{w: v \mapsto w} (\textup{label on edge } v \mapsto w)$, the product of the labels on the edges directed out of $v$; and
\item choosing homogeneous polynomials of degree $\deg p_v(v)$ that satisfy the GKM conditions for each other $p_v(u)$.
\end{enumerate}
The equivariant class in Figure \ref{examples of classes} is the canonical class corresponding to $s_1$.  Figure \ref{examples of canonical} gives examples of canonical classes in two different algebraic varieties.  In the case of the flag variety, the canonical class $p_w$ is the same as the class $\xi^w$ defined by Kostant and Kumar in \cite{KK}.
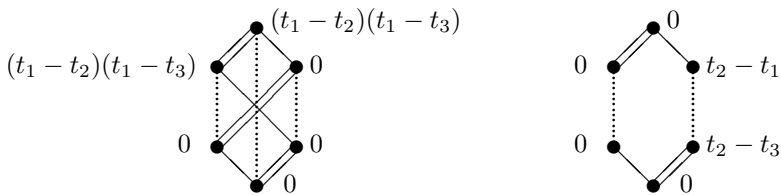
\begin{figure}[h]
\begin{picture}(360,60)(90,-30)
\put(200,-30){\circle*{5}}
\put(185,-15){\circle*{5}}
\put(215,-15){\circle*{5}}
\put(185,15){\circle*{5}}
\put(215,15){\circle*{5}}
\put(200,30){\circle*{5}}

\put(200,-30){\line(-1,1){15}}
\put(215,-15){\line(-1,1){30}}
\put(215,15){\line(-1,1){15}}

\put(201,-31){\line(1,1){15}}
\put(199,-29){\line(1,1){15}}
\put(186,-16){\line(1,1){30}}
\put(184,-14){\line(1,1){30}}
\put(186,14){\line(1,1){15}}
\put(184,16){\line(1,1){15}}

\multiput(185,-15)(0,2){15}{\circle*{1}}
\multiput(200,-30)(0,2){30}{\circle*{1}}
\multiput(215,-15)(0,2){15}{\circle*{1}}

\put(220,-17){$0$}
\put(220,13){$0$}
\put(170,-17){$0$}
\put(105,13){$(t_1-t_2)(t_1-t_3)$}
\put(205,30){$(t_1-t_2)(t_1-t_3)$}
\put(210,-32){$0$}

\put(350,-30){\circle*{5}}
\put(335,-15){\circle*{5}}
\put(365,-15){\circle*{5}}
\put(335,15){\circle*{5}}
\put(365,15){\circle*{5}}
\put(350,30){\circle*{5}}

\put(350,-30){\line(-1,1){15}}
\put(365,15){\line(-1,1){15}}

\put(351,-31){\line(1,1){15}}
\put(349,-29){\line(1,1){15}}
\put(336,14){\line(1,1){15}}
\put(334,16){\line(1,1){15}}

\multiput(335,-15)(0,2){15}{\circle*{1}}
\multiput(365,-15)(0,2){15}{\circle*{1}}

\put(370,-17){$t_2-t_3$}
\put(370,13){$t_2-t_1$}
\put(320,-17){$0$}
\put(320,13){$0$}
\put(355,30){$0$}
\put(360,-32){$0$}
\end{picture}
\caption{Two examples of canonical classes} \label{examples of canonical}
\end{figure}

We say a few words about why canonical classes form a basis.  Complete proofs are provided by V.\ Guillemin and C.\ Zara in \cite{GZ1} or in \cite{T3}.  The fact that the moment graph is directed gives a partial order on the fixed points of the GKM space.  Condition \eqref{zero below} ensures that the canonical classes are linearly independent, since each has a different minimal nonzero entry.  Any equivariant class $p$ that vanishes on all $u \leq v$ must satisfy $p(v) = cp_v(v)$ for some polynomial $c$ by Condition \eqref{product on bottom vertex} and the GKM conditions.  Using this fact together with any order subordinate to the partial order on the fixed points, an arbitrary equivariant class can be expressed as an element in the span of the canonical classes.

V.\ Guillemin and C.\ Zara analyzed the combinatorial circumstances under which canonical classes are guaranteed to exist in \cite{GZ2}.  These conditions are satisfied by flag varieties, Schubert varieties, Grassmannians, and many other varieties of geometric and combinatorial interest.  \cite{GZ2} also showed that canonical classes are unique if the GKM space is also a Palais-Smale manifold, which implies that if the gradient flow up from one fixed point intersects the gradient flow down from an adjacent fixed point then the intersection is a $T$-invariant $2$-dimensional sphere \cite{K1}.\footnote{In general, a Morse function is Palais-Smale if its stable and unstable manifolds intersect transversely.}  In fact, the GKM space need only be an algebraic variety that satisfies a combinatorial analogue of the Palais-Smale condition: if the number of down-edges strictly increases along each upward path in the moment graph, namely for each edge $v \mapsto u$ the cardinalities $|\{w: v \mapsto w\}| > |\{w: u \mapsto w\}|$, then canonical classes are unique \cite{T3}.  This condition is satisfied by flag varieties, Schubert varieties, and Grassmannians, though there are natural--even smooth!--algebraic varieties for which it does not hold.  In fact, the variety on the right in Figure \ref{examples of canonical} does not satisfy this combinatorial Palais-Smale condition.  (It is the toric variety associated to the decomposition of Weyl chambers.)

This leads us to the main result of this section (\cite{T3} has a proof).  We remark that though it is algebraically natural to write Schubert cells as $[Bw]$, our claims instead use Schubert cells $[B^-w]$, for the subgroup $B^-$ of {\em lower} triangular matrices.

\begin{proposition}
The canonical classes $p_v$ are the equivariant Schubert classes induced from the Schubert cells $[B^-w]$.
\end{proposition}

The difference between using Schubert cells $[Bw]$ and $[B^-w]$ is precisely the difference between defining canonical classes on the flow-up from a vertex (so that they are zero below the vertex) or on the flow-down (so they are zero above).  A canonical class defined using the flow-down is shown in Figure \ref{flow down example}.  It is the Schubert class that is Poincar\'{e} dual to the homology class of the Schubert variety $\overline{[Bs_1s_2]}$.
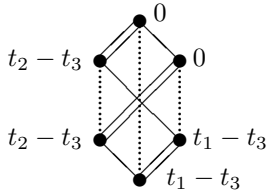
\begin{figure}[h]
\begin{picture}(360,60)(20,-30)
\put(200,-30){\circle*{5}}
\put(185,-15){\circle*{5}}
\put(215,-15){\circle*{5}}
\put(185,15){\circle*{5}}
\put(215,15){\circle*{5}}
\put(200,30){\circle*{5}}

\put(200,-30){\line(-1,1){15}}
\put(215,-15){\line(-1,1){30}}
\put(215,15){\line(-1,1){15}}

\put(201,-31){\line(1,1){15}}
\put(199,-29){\line(1,1){15}}
\put(186,-16){\line(1,1){30}}
\put(184,-14){\line(1,1){30}}
\put(186,14){\line(1,1){15}}
\put(184,16){\line(1,1){15}}

\multiput(185,-15)(0,2){15}{\circle*{1}}
\multiput(200,-30)(0,2){30}{\circle*{1}}
\multiput(215,-15)(0,2){15}{\circle*{1}}

\put(220,-17){$t_1-t_3$}
\put(220,13){$0$}
\put(150,-17){$t_2-t_3$}
\put(150,13){$t_2-t_3$}
\put(205,30){$0$}
\put(210,-32){$t_1-t_3$}
\end{picture}
\caption{A canonical class defined by the flow-down} \label{flow down example}
\end{figure}

A nice {\bf exercise} for the reader is to compute all the canonical classes for the flag variety $GL_3(\C)/B$ and for the flag variety of type $B_2$.

\section{Two permutation actions on the equivariant cohomology of flag varieties}

The permutation group acts on itself in two ways, by left multiplication and by right multiplication.  These two actions induce graph automorphisms of the moment graph of the flag variety.  The amazing fact is that these graph automorphisms in fact induce actions of the permutation group on the equivariant cohomology of the flag variety.  The first part of this section defines these two permutation actions.  The second asks---and answers---the natural Schubert calculus question: what is the image of an equivariant Schubert class under the action of a simple transposition in terms of the basis of Schubert classes?  Both actions give rise to `divided difference operators', which are degree-lowering rational operators on equivariant cohomology studied in the third part of the section.  Divided difference operators are extremely useful for computational purposes because they are well-adapted to inductive arguments.  For instance, they have been used to identify certain structure constants in the cohomology ring \cite{BGG}, \cite{KK}, to determine the localizations of equivariant Schubert classes \cite{Bi}, and to generate all the Schubert classes from one particular class \cite{BGG}.  The section concludes with a small result on localizations that follows from properties of the divided difference operators. 

\subsection{The permutation actions} The permutation group $S_n$ acts on polynomials $\C[t_1, \ldots, t_n]$ by permuting indices, and $S_n$ acts on itself either by left or right multiplication.  Together, these actions give two different permutation actions on the equivariant cohomology of flag varieties.  In this section we discuss both actions.  It turns out that the one that is simpler to compute is in fact a more complicated action on equivariant cohomology.

We begin by defining the `dot' action of the permutation group.  For each polynomial $f \in \C[t_1, \ldots, t_n]$ and permutation $w \in S_n$ is a permutation, define $wf = f(t_{w(1)}, t_{w(2)}, \ldots, t_{w(n)})$.  The first action of $w \in S_n$ on an equivariant class $p \in H^*_T(GL_n(\C)/B)$ is given locally for each fixed point $v$ by
\[(w \cdot p)(v) = wp(w^{-1}v).\]
This action was defined in \cite{T3} and geometrically by M.\ Brion in \cite{B}.  It is easiest to visualize if we restrict our attention to the action of simple transpositions $s_i = (i,i+1)$.  In that case, the action does two things simultaneously:
\begin{itemize}
\item exchanges polynomials on either side of an edge labeled $t_i - t_{i+1}$, and
\item exchanges the variables $t_i$ and $t_{i+1}$ in each polynomial.
\end{itemize}
Figure \ref{dot action example} gives an example of the action of $s_1$ on a particular Schubert class.  The edge-labels are as in Figure \ref{A2 flags}, so single-lines are edges labeled $t_1-t_2$.
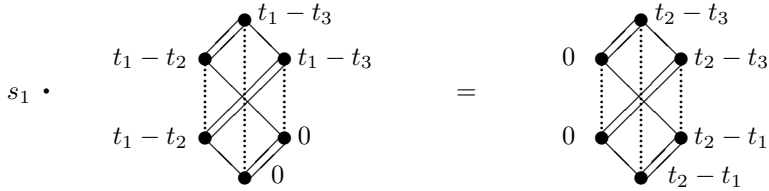
\begin{figure}[h]
\begin{picture}(360,60)(70,-30)
\put(200,-30){\circle*{5}}
\put(185,-15){\circle*{5}}
\put(215,-15){\circle*{5}}
\put(185,15){\circle*{5}}
\put(215,15){\circle*{5}}
\put(200,30){\circle*{5}}

\put(200,-30){\line(-1,1){15}}
\put(215,-15){\line(-1,1){30}}
\put(215,15){\line(-1,1){15}}

\put(201,-31){\line(1,1){15}}
\put(199,-29){\line(1,1){15}}
\put(186,-16){\line(1,1){30}}
\put(184,-14){\line(1,1){30}}
\put(186,14){\line(1,1){15}}
\put(184,16){\line(1,1){15}}

\multiput(185,-15)(0,2){15}{\circle*{1}}
\multiput(200,-30)(0,2){30}{\circle*{1}}
\multiput(215,-15)(0,2){15}{\circle*{1}}

\put(220,-17){$0$}
\put(220,13){$t_1-t_3$}
\put(150,-17){$t_1-t_2$}
\put(150,13){$t_1-t_2$}
\put(205,30){$t_1-t_3$}
\put(210,-32){$0$}

\put(350,-30){\circle*{5}}
\put(335,-15){\circle*{5}}
\put(365,-15){\circle*{5}}
\put(335,15){\circle*{5}}
\put(365,15){\circle*{5}}
\put(350,30){\circle*{5}}

\put(350,-30){\line(-1,1){15}}
\put(365,-15){\line(-1,1){30}}
\put(365,15){\line(-1,1){15}}

\put(351,-31){\line(1,1){15}}
\put(349,-29){\line(1,1){15}}
\put(336,-16){\line(1,1){30}}
\put(334,-14){\line(1,1){30}}
\put(336,14){\line(1,1){15}}
\put(334,16){\line(1,1){15}}

\multiput(335,-15)(0,2){15}{\circle*{1}}
\multiput(350,-30)(0,2){30}{\circle*{1}}
\multiput(365,-15)(0,2){15}{\circle*{1}}

\put(370,-17){$t_2-t_1$}
\put(370,13){$t_2-t_3$}
\put(320,-17){$0$}
\put(320,13){$0$}
\put(355,30){$t_2-t_3$}
\put(360,-32){$t_2-t_1$}

\put(280,0){$=$}
\put(110,0){$s_1$}
\put(125,3){\circle*{2}}
\end{picture}
\caption{The class $s_1 \cdot p_{s_1}$ in $H^*_T(GL_3(\C)/B)$} \label{dot action example}
\end{figure}
A good {\bf exercise} for the reader is to compute the class $s_2 \cdot p_{s_1}$.

Next we describe the `star' action.  The general formula is noticeably simpler: if $w \in S_n$ and $p \in H^*_T(GL_n(\C)/B)$ then for each fixed point $v$ 
\[(p * w)(v) = p(vw).\]
In this case, the action of the simple transposition exchanges polynomials on either side of an edge $w \leftrightarrow ws_i$.  Note that {\em there is no permutation action on the variables}.  This action is simpler to write.  However, it requires calculation to determine which fixed points differ by {\em right} multiplication by $s_i$.  Figure \ref{star action example} shows the action in the case $p_{s_1} * s_1$.  The edges $w \leftrightarrow ws_i$ are marked with large dots.
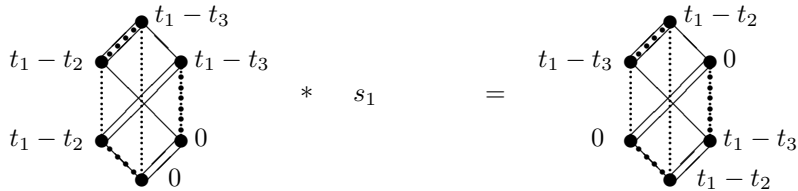
\begin{figure}[h]
\begin{picture}(360,60)(70,-30)
\put(150,-30){\circle*{5}}
\put(135,-15){\circle*{5}}
\put(165,-15){\circle*{5}}
\put(135,15){\circle*{5}}
\put(165,15){\circle*{5}}
\put(150,30){\circle*{5}}

\put(150,-30){\line(-1,1){15}}
\put(165,-15){\line(-1,1){30}}
\put(165,15){\line(-1,1){15}}

\put(151,-31){\line(1,1){15}}
\put(149,-29){\line(1,1){15}}
\put(136,-16){\line(1,1){30}}
\put(134,-14){\line(1,1){30}}
\put(136,14){\line(1,1){15}}
\put(134,16){\line(1,1){15}}

\multiput(135,-15)(0,2){15}{\circle*{1}}
\multiput(150,-30)(0,2){30}{\circle*{1}}
\multiput(165,-15)(0,2){15}{\circle*{1}}

\put(170,-17){$0$}
\put(170,13){$t_1-t_3$}
\put(100,-17){$t_1-t_2$}
\put(100,13){$t_1-t_2$}
\put(155,30){$t_1-t_3$}
\put(160,-32){$0$}

\multiput(150,-30)(-3,3){5}{\circle*{2}}
\multiput(135,15)(3,3){5}{\circle*{2}}
\multiput(165, -15)(0,4){8}{\circle*{2}}

\put(350,-30){\circle*{5}}
\put(335,-15){\circle*{5}}
\put(365,-15){\circle*{5}}
\put(335,15){\circle*{5}}
\put(365,15){\circle*{5}}
\put(350,30){\circle*{5}}

\put(350,-30){\line(-1,1){15}}
\put(365,-15){\line(-1,1){30}}
\put(365,15){\line(-1,1){15}}

\put(351,-31){\line(1,1){15}}
\put(349,-29){\line(1,1){15}}
\put(336,-16){\line(1,1){30}}
\put(334,-14){\line(1,1){30}}
\put(336,14){\line(1,1){15}}
\put(334,16){\line(1,1){15}}

\multiput(335,-15)(0,2){15}{\circle*{1}}
\multiput(350,-30)(0,2){30}{\circle*{1}}
\multiput(365,-15)(0,2){15}{\circle*{1}}

\put(370,-17){$t_1-t_3$}
\put(370,13){$0$}
\put(320,-17){$0$}
\put(300,13){$t_1-t_3$}
\put(355,30){$t_1-t_2$}
\put(360,-32){$t_1-t_2$}

\multiput(350,-30)(-3,3){5}{\circle*{2}}
\multiput(335,15)(3,3){5}{\circle*{2}}
\multiput(365, -15)(0,4){8}{\circle*{2}}

\put(280,0){$=$}
\put(230,0){$s_1$}
\put(210,-2){*}
\end{picture}
\caption{The class $p_{s_1} * s_1$ in $H^*_T(GL_3(\C)/B)$} \label{star action example}
\end{figure}
A good {\bf exercise} for the reader is to compute the class $p_{s_1} * s_2$.

The star action was first defined and studied by B.\ Kostant and S.\ Kumar in \cite{KK}.  Among many other results in that very substantial paper, they show:

\begin{theorem}
(Kostant-Kumar) The star action of each $w \in S_n$ is a well-defined $\C[t_1, \ldots, t_n]$-algebra automorphism $ * w: H^*_T(GL_n/B) \rightarrow H^*_T(GL_n/B)$.
\end{theorem}

A.\ Knutson describes both actions on the flag variety in the unfortunately unpublished \cite{K2}.  There, he notes that the dot action is induced from a left action of the permutation group on the fixed points (also the permutation group) while the star action is induced from a right action on the permutation group.  Thus, he refers to them as the left and right actions, a policy we very much approve.  

The dot action is well-defined but is not a $\C[t_1, \ldots, t_n]$-algebra automorphism.

\begin{theorem}
The dot action of each $w \in S_n$ is a twisted $\C[t_1, \ldots, t_n]$-algebra automorphism $w \cdot : H^*_T(GL_n(\C)/B) \rightarrow H^*_T(GL_n(\C)/B)$, in the sense that if $c \in \C[t_1, \ldots, t_n]$ and $p \in H^*_T(GL_n(\C)/B)$ then $w \cdot (cp) = (wc) (w \cdot p)$.
\end{theorem}

A proof is in \cite{T3}.  Naturally, these results extend to general flag varieties as well.  More surprisingly, the dot action---but not the star action---generalizes to all Grassmannians $G/P$, including the Grassmannian of $k$-dimensional planes in $\C^n$.  This is shown in \cite{T2} and from a geometric perspective in \cite{B}.  Figure \ref{C2 action examples} shows the dot action of $s_2$ on the class $p_{s_2}$ in the flag variety for $C_2$.  The reader is left the {\bf exercise} of computing the star action of $s_2$ on the same class.
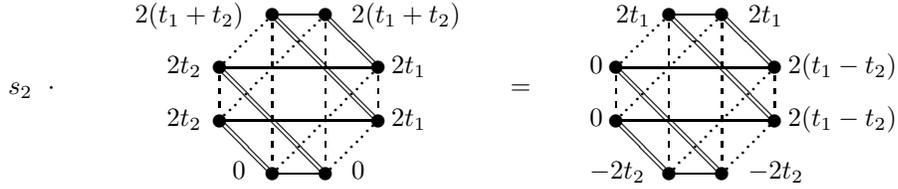
\begin{figure}[h]
\begin{picture}(360,70)(80,-35)
\put(200,-30){\circle*{5}}
\put(220,-30){\circle*{5}}
\put(180,-10){\circle*{5}}
\put(240,-10){\circle*{5}}
\put(180,10){\circle*{5}}
\put(240,10){\circle*{5}}
\put(200,30){\circle*{5}}
\put(220,30){\circle*{5}}

\put(199,-30){\line(-1,1){20}}
\put(219,-30){\line(-1,1){40}}
\put(239,-10){\line(-1,1){40}}
\put(239,10){\line(-1,1){20}}
\put(201,-30){\line(-1,1){20}}
\put(221,-30){\line(-1,1){40}}
\put(241,-10){\line(-1,1){40}}
\put(241,10){\line(-1,1){20}}

\put(200,-30){\line(1,0){20}}
\put(180,-10){\line(1,0){60}}
\put(180,10){\line(1,0){60}}
\put(200,30){\line(1,0){20}}

\multiput(220,-30)(2,2){10}{\circle*{1}}
\multiput(200,-30)(2,2){20}{\circle*{1}}
\multiput(180,-10)(2,2){20}{\circle*{1}}
\multiput(180,10)(2,2){10}{\circle*{1}}

\multiput(180,-10)(0,5){4}{\line(0,1){2}}
\multiput(200,-30)(0,5){12}{\line(0,1){2}}
\multiput(220,-30)(0,5){12}{\line(0,1){2}}
\multiput(240,-10)(0,5){4}{\line(0,1){2}}

\put(185,-32){$0$}
\put(230,-32){$0$}
\put(160,-12){$2t_2$}
\put(245,-12){$2t_1$}
\put(160,8){$2t_2$}
\put(245,8){$2t_1$}
\put(148,27){$2(t_1+t_2)$}
\put(230,27){$2(t_1+t_2)$}

\put(350,-30){\circle*{5}}
\put(370,-30){\circle*{5}}
\put(330,-10){\circle*{5}}
\put(390,-10){\circle*{5}}
\put(330,10){\circle*{5}}
\put(390,10){\circle*{5}}
\put(350,30){\circle*{5}}
\put(370,30){\circle*{5}}

\put(349,-30){\line(-1,1){20}}
\put(369,-30){\line(-1,1){40}}
\put(389,-10){\line(-1,1){40}}
\put(389,10){\line(-1,1){20}}
\put(351,-30){\line(-1,1){20}}
\put(371,-30){\line(-1,1){40}}
\put(391,-10){\line(-1,1){40}}
\put(391,10){\line(-1,1){20}}

\put(350,-30){\line(1,0){20}}
\put(330,-10){\line(1,0){60}}
\put(330,10){\line(1,0){60}}
\put(350,30){\line(1,0){20}}

\multiput(370,-30)(2,2){10}{\circle*{1}}
\multiput(350,-30)(2,2){20}{\circle*{1}}
\multiput(330,-10)(2,2){20}{\circle*{1}}
\multiput(330,10)(2,2){10}{\circle*{1}}

\multiput(330,-10)(0,5){4}{\line(0,1){2}}
\multiput(350,-30)(0,5){12}{\line(0,1){2}}
\multiput(370,-30)(0,5){12}{\line(0,1){2}}
\multiput(390,-10)(0,5){4}{\line(0,1){2}}

\put(320,-32){$-2t_2$}
\put(380,-32){$-2t_2$}
\put(320,-12){$0$}
\put(395,-12){$2(t_1-t_2)$}
\put(320,8){$0$}
\put(395,8){$2(t_1-t_2)$}
\put(330,27){$2t_1$}
\put(380,27){$2t_1$}

\put(100,0){$s_2$}
\put(115,0){$\cdot$}
\put(290,0){$=$}
\end{picture}
\caption{An example of the dot action for flags of type $C_2$} \label{C2 action examples}
\end{figure}

\subsection{Formulas for dot and star}

The natural Schubert calculus question is: how do the two permutation actions interact with the basis of Schubert classes?  We answer this for the action of a simple transposition on Schubert classes.  This is the foundation for our later study of the permutation representations on the equivariant cohomology of the flag variety.

\begin{theorem} \label{dot action formula}
If $s_i \in S_n$ is a simple reflection and $p_w \in H^*_T(GL_n(\C)/B)$ is a Schubert class then
\[s_i \cdot p_w = \left\{ \begin{array}{ll} p_w - (t_i - t_{i+1})p_{s_iw} & \textup{ if } s_iw < w \textup{ in the moment graph and} \\		p_w & \textup{ otherwise.} \end{array} \right.\]
\end{theorem}

This is shown in \cite{T3} and geometrically in \cite{B}.  The same formula holds for the action of simple reflections on Grassmannians \cite{T2}, \cite{B}.  We sketch the proof here.

\begin{proof} Sketch: The class $s_i \cdot p_w$ is homogeneous of degree $\ell(w)$ so it can be written as a combination $s_i \cdot p_w = \sum c_vp_v$ where the sum is taken over permutations $v$ whose length is at most $\ell(w)$.  The action of $s_i$ exchanges polynomials $p_w(v)$ and $p_w(s_iv)$.  The lengths of these fixed points satisfy $\ell(v) = \ell(s_iv) \pm 1$.  Moreover, the canonical class $p_w$ by definition has $p_w(v) = 0$ if $v \not > w$.  We conclude that the sum in $s_i \cdot p_w = \sum c_vp_v$ need only be taken over permutations $v$ of length $\ell(w)$ as well as perhaps $v = s_iw$.  This leaves a small number of possible nonzero coefficients $c_v$ which are computed directly using the combinatorics of the permutation group. \end{proof}

As J.\ Stembridge observed, the star action ``does great violence to the Schubert classes."  This is evidenced in the formula for the action of a simple transposition.

\begin{theorem}
(Kostant-Kumar) If $s_i \in S_n$ is a simple reflection and $p_w \in H^*_T(GL_n(\C)/B)$ is a Schubert class then
\[p_w * s_i= \left\{ \begin{array}{ll} p_w - (t_{w(i)}-t_{w(i+1)})p_{ws_i} + & \\
\hspace{.5in} \displaystyle \sum_{\footnotesize \begin{array}{c}  (jk) s.t.  \\ \ell(ws_i(jk)) = \ell(w) \end{array}} \langle t_i-t_{i+1}, t_j - t_k \rangle p_{ws_i(jk)} & \textup{ if } ws_i<w \textup{ and} \\ p_w & \textup{ otherwise,} \end{array} \right.\]
where $\langle t_{i_1} -t_{i_2}, t_j - t_k \rangle = \frac{(jk)(t_{i_1}-t_{i_2}) - (t_{i_1}-t_{i_2})}{t_j-t_k}$.
\end{theorem}

This appears in \cite[Lemma 5.9]{KK}.  A very beautiful proof is provided in \cite[Corollary to Proposition 3]{K2}.  The sketch of the proof is similar to the previous but the calculations of the coefficients are more complicated.  We remark that the coefficient of $p_w$ in $p_w * s_i$ is $-1$ and that in the flag variety $GL_n(\C)/B$ the terms $\langle t_i - t_{i+1}, t_j - t_k \rangle$ are $0$, $1$, or $-1$ as long as $(jk) \neq (i,i+1)$.

This theorem also applies to all types.  The coefficient of $p_{ws_i}$ is $-w(\alpha_i)$ and the coefficient of $p_{ws_is_{\beta}}$ is $\langle \alpha_i, \beta \rangle$ as long as $\alpha_i \neq \beta$.

\subsection{Divided difference operators}

Divided difference operators are a family of `discrete derivatives' defined on the (ordinary and equivariant) cohomology of the flag variety.  They were first defined by Bernstein-Gelfand-Gelfand and Demazure in \cite{BGG} and \cite{D} to express the geometric basis of Schubert classes in terms of an algebraic description of the cohomology of the flag variety.  This algebraic construction says that $H^*(GL_n(\C)/B) = \mathbb{Z}[x_1, \ldots, x_n]/I$ where $I$ is the ideal generated by symmetric polynomials with no constant term.  If the flag variety is written as the collection of nested subspaces $V_{\bullet} = V_1 \subseteq V_2 \subseteq \cdots \subseteq V_{n-1} \subseteq \C^n$ then it turns out that each $-x_i$ represents  the Chern class of the line bundle defined by $V_{i}/V_{i-1} \rightarrow V_{\bullet}$.  

As before, the permutation group acts on $\mathbb{Z}[x_1, \ldots, x_n]$ by permuting variables.  The $i^{th}$ divided difference operator $\partial_i$ acts on the polynomial $f$ by
\[ \partial_i f = \frac{f - s_i f}{x_i - x_{i+1}}.\]
For instance, we have 
\[\partial_2(x_1^2x_2) = \frac{x_1^2x_2 - x_1^2x_3}{x_2-x_3} = x_1^2\]
and 
\[\partial_1(x_1^2) = \frac{x_1^2 - x_2^2}{x_1-x_2} = x_1+x_2.\]
Note that if $f$ is a polynomial that is symmetric in the variables $x_i$ and $x_{i+1}$ then $\partial_i f = 0$.  An {\bf exercise} for the reader is to show that the ideal $I$ is exactly the intersection $\bigcap_{i=1}^n \ker(\partial_i)$.  Consequently, divided difference operators are well-defined on the cohomology of the flag variety, and in fact studying divided difference operators amounts to studying the cohomology of the flag variety.

The goal of \cite{BGG} and \cite{D} was to choose good polynomial representatives of Schubert classes in $\mathbb{Z}[x_1, \ldots, x_n]/I$, called Schubert polynomials and denoted $\mathfrak{S}_w$.    

%The reader may confirm as an {\bf exercise} that applying all possible sequences of the divided difference operators $\partial_1$ and $\partial_2$ to the polynomial $\frac{1}{6}(x_1-x_2)(x_1-x_3)(x_2-x_3)$ generates six distinct polynomials, one each of degree three and zero, and two each of degrees one and two.  These match the degrees and numbers of the Schubert classes as identified by the GKM method.
We highlight two key properties of divided difference operators.  First, the divided difference operator should satisfy the braid relations (and nil-Coxeter relations), meaning that if $w = s_{i_1} \cdots s_{i_k}$ and $w = s_{j_1} \cdots s_{j_k}$ are any two minimal-length factorizations of the permutation $w$ then
\begin{equation} \label{first div diff prop}
\partial_{i_1} \cdots \partial_{i_k} = \partial_{j_1} \cdots \partial_{j_k}.
\end{equation}
We subsume the nil-Coxeter relation into the braid relations; it asserts that $\partial_i^2 = 0$ for each $i$.  The second key property is that divided difference operators respect Schubert classes in the sense that \begin{equation} \label{second div diff prop}
\partial_i \mathfrak{S}_w = \left\{ \begin{array}{ll} \mathfrak{S}_{ws_i} & \textup{ if } ws_i < w \textup{ and } \\ 0 & \textup{ otherwise.} \end{array} \right.
\end{equation}
In practice, this second property may involve Schubert polynomials $\mathfrak{S}_w$ or Schubert classes $p_w$; it may involve multiplication on the right $ws_i$ or on the left $s_iw$ by the simple transposition.  Nonetheless, a divided difference operator should satisfy an equation very similar to this.

Divided difference operators appear in many different contexts.  Algebraically, they have been used to obtain structure constants of $H^*(GL_n(\C)/B)$ with respect to the basis of Schubert classes in \cite{BGG} and \cite{KK}.  They have been extended to very general topological settings in \cite{BE} and geometric settings in \cite{B}.  Combinatorists have used divided difference operators to answer the question: of the many possible choices for a polynomial within the coset of a Schubert class, which is the `best'?  Key properties that seem useful are:
\begin{enumerate}
\item the Schubert polynomial $\mathfrak{S}_w$ is homogeneous of degree $\ell(w)$;
\item the Schubert polynomial $\mathfrak{S}_w$ is positive in the $x_i$;
\item the Schubert polynomial $\mathfrak{S}_w$ respects divided difference operators in the sense that
\[\partial_i \mathfrak{S}_w = \left\{ \begin{array}{ll} \mathfrak{S}_{ws_i} & \textup{ if } ws_i < w \textup{ and} \\  0 & \textup{ otherwise;} \end{array} \right.\]
\item the product of Schubert polynomials $\mathfrak{S}_u \mathfrak{S}_v = \sum_w c_{uv}^w \mathfrak{S}_w$ where $c_{uv}^w$ are the structure constants in $H^*(GL_n(\C)/B)$.
\end{enumerate}
Lascoux and Schutzenberger proved that if the highest-degree Schubert polynomial is chosen to be $\prod_{i=1}^n x_i^{n-i}$ then all of these properties are satisfied simultaneously \cite{LS}.  Unlike most results in this paper, these do not entirely generalize to other types.  In fact, Fomin and Kirillov showed that all four properties cannot be satisfied at once in other types \cite{FK}.  (If one requires positivity in the simple roots---for instance, the terms $x_i - x_{i+1}$ for the flag variety $GL_n(\C)/B$---then Properties (1)--(3) are satisfied by the polynomials proposed by Bernstein-Gelfand-Gelfand, whose top-degree polynomial was $\frac{1}{n!}\prod_{1 \leq i< j \leq n} (x_i - x_j)$.)

There are two divided difference operators on the equivariant cohomology of the flag variety, each  related to one of the permutation actions described previously.

The first is the dot divided difference operator $\delta_i$, defined in \cite{T3}, \cite{K2} and geometrically in \cite{B}, though misidentified in the latter two papers.  It resembles the traditional divided difference operator closely: if $p \in H^*_T(GL_n(\C)/B)$ then
\[ \delta_i p = \frac{p - s_i \cdot p}{t_i - t_{i+1}}.\]
Note that $p$ is no longer a polynomial but a collection of polynomials and that $s_i$ acts more complexly than just on the variables.  Figure \ref{new div diff figure} gives an example.
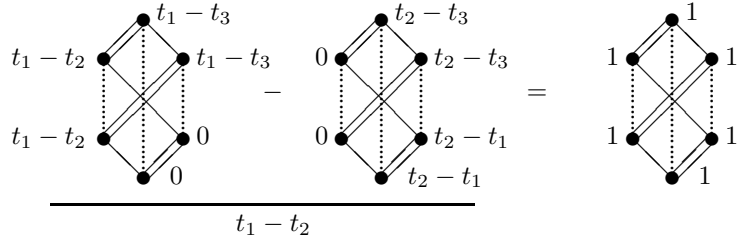
\begin{figure}[h]
\begin{picture}(360,80)(80,-50)
\put(180,-30){\circle*{5}}
\put(165,-15){\circle*{5}}
\put(195,-15){\circle*{5}}
\put(165,15){\circle*{5}}
\put(195,15){\circle*{5}}
\put(180,30){\circle*{5}}

\put(180,-30){\line(-1,1){15}}
\put(195,-15){\line(-1,1){30}}
\put(195,15){\line(-1,1){15}}

\put(181,-31){\line(1,1){15}}
\put(179,-29){\line(1,1){15}}
\put(166,-16){\line(1,1){30}}
\put(164,-14){\line(1,1){30}}
\put(166,14){\line(1,1){15}}
\put(164,16){\line(1,1){15}}

\multiput(165,-15)(0,2){15}{\circle*{1}}
\multiput(180,-30)(0,2){30}{\circle*{1}}
\multiput(195,-15)(0,2){15}{\circle*{1}}

\put(200,-17){$0$}
\put(200,13){$t_1-t_3$}
\put(130,-17){$t_1-t_2$}
\put(130,13){$t_1-t_2$}
\put(185,30){$t_1-t_3$}
\put(190,-32){$0$}

\put(270,-30){\circle*{5}}
\put(255,-15){\circle*{5}}
\put(285,-15){\circle*{5}}
\put(255,15){\circle*{5}}
\put(285,15){\circle*{5}}
\put(270,30){\circle*{5}}

\put(270,-30){\line(-1,1){15}}
\put(285,-15){\line(-1,1){30}}
\put(285,15){\line(-1,1){15}}

\put(271,-31){\line(1,1){15}}
\put(269,-29){\line(1,1){15}}
\put(256,-16){\line(1,1){30}}
\put(254,-14){\line(1,1){30}}
\put(256,14){\line(1,1){15}}
\put(254,16){\line(1,1){15}}

\multiput(255,-15)(0,2){15}{\circle*{1}}
\multiput(270,-30)(0,2){30}{\circle*{1}}
\multiput(285,-15)(0,2){15}{\circle*{1}}

\put(290,-17){$t_2-t_1$}
\put(290,13){$t_2-t_3$}
\put(245,-17){$0$}
\put(245,13){$0$}
\put(275,30){$t_2-t_3$}
\put(280,-32){$t_2-t_1$}

\put(380,-30){\circle*{5}}
\put(365,-15){\circle*{5}}
\put(395,-15){\circle*{5}}
\put(365,15){\circle*{5}}
\put(395,15){\circle*{5}}
\put(380,30){\circle*{5}}

\put(380,-30){\line(-1,1){15}}
\put(395,-15){\line(-1,1){30}}
\put(395,15){\line(-1,1){15}}

\put(381,-31){\line(1,1){15}}
\put(379,-29){\line(1,1){15}}
\put(366,-16){\line(1,1){30}}
\put(364,-14){\line(1,1){30}}
\put(366,14){\line(1,1){15}}
\put(364,16){\line(1,1){15}}

\multiput(365,-15)(0,2){15}{\circle*{1}}
\multiput(380,-30)(0,2){30}{\circle*{1}}
\multiput(395,-15)(0,2){15}{\circle*{1}}

\put(400,-17){$1$}
\put(400,13){$1$}
\put(355,-17){$1$}
\put(355,13){$1$}
\put(385,30){$1$}
\put(390,-32){$1$}

\put(225,0){$-$}
\put(325,0){$=$}
\put(145,-40){\line(1,0){160}}
\put(215,-50){$t_1-t_2$}
\end{picture}
\caption{Calculating $\delta_1(p_{s_1})$} \label{new div diff figure}
\end{figure}

The star action is too rough on Schubert classes to give a divided difference operator directly.  To correct for this, Kostant-Kumar defined a divided difference operator on the localizations: for each $p \in H^*_T(GL_n(\C)/B)$ let
\[(\partial_i p) (v) = \frac{p(v) - p(vs_i)}{-t_{v(i)} + t_{v(i+1)}}.\]
A very nice presentation of some of Kostant-Kumar's results is also in \cite{K2}, which notes that the divided difference operator can be written
\[\partial_i p = \frac{p - p * s_i}{-c_i}\]
where $c_i$ is the Chern class of the $i^{th}$ elementary line bundle on the flag variety.  Figure \ref{BGG div diff figure} gives an example.

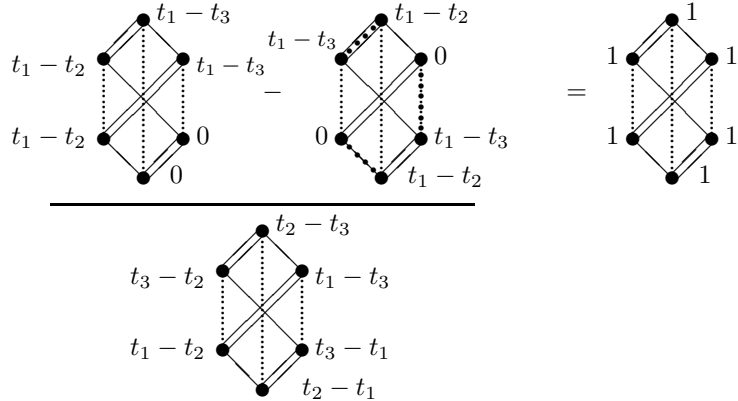
\begin{figure}[h]
\begin{picture}(360,140)(80,-110)
\put(225,-110){\circle*{5}}
\put(210,-95){\circle*{5}}
\put(240,-95){\circle*{5}}
\put(210,-65){\circle*{5}}
\put(240,-65){\circle*{5}}
\put(225,-50){\circle*{5}}

\put(225,-110){\line(-1,1){15}}
\put(240,-95){\line(-1,1){30}}
\put(240,-65){\line(-1,1){15}}

\put(226,-111){\line(1,1){15}}
\put(224,-109){\line(1,1){15}}
\put(211,-96){\line(1,1){30}}
\put(209,-94){\line(1,1){30}}
\put(209,-64){\line(1,1){15}}
\put(211,-66){\line(1,1){15}}

\multiput(210,-95)(0,2){15}{\circle*{1}}
\multiput(225,-110)(0,2){30}{\circle*{1}}
\multiput(240,-95)(0,2){15}{\circle*{1}}

\put(245,-97){$t_3-t_1$}
\put(245,-70){$t_1-t_3$}
\put(175,-97){$t_1-t_2$}
\put(175,-70){$t_3-t_2$}
\put(230,-50){$t_2-t_3$}
\put(240,-112){$t_2-t_1$}

\put(180,-30){\circle*{5}}
\put(165,-15){\circle*{5}}
\put(195,-15){\circle*{5}}
\put(165,15){\circle*{5}}
\put(195,15){\circle*{5}}
\put(180,30){\circle*{5}}

\put(180,-30){\line(-1,1){15}}
\put(195,-15){\line(-1,1){30}}
\put(195,15){\line(-1,1){15}}

\put(181,-31){\line(1,1){15}}
\put(179,-29){\line(1,1){15}}
\put(166,-16){\line(1,1){30}}
\put(164,-14){\line(1,1){30}}
\put(166,14){\line(1,1){15}}
\put(164,16){\line(1,1){15}}

\multiput(165,-15)(0,2){15}{\circle*{1}}
\multiput(180,-30)(0,2){30}{\circle*{1}}
\multiput(195,-15)(0,2){15}{\circle*{1}}

\put(200,-17){$0$}
\put(200,10){\small $t_1-t_3$}
\put(130,-17){$t_1-t_2$}
\put(130,10){$t_1-t_2$}
\put(185,30){$t_1-t_3$}
\put(190,-32){$0$}

\put(270,-30){\circle*{5}}
\put(255,-15){\circle*{5}}
\put(285,-15){\circle*{5}}
\put(255,15){\circle*{5}}
\put(285,15){\circle*{5}}
\put(270,30){\circle*{5}}

\put(270,-30){\line(-1,1){15}}
\put(285,-15){\line(-1,1){30}}
\put(285,15){\line(-1,1){15}}

\put(271,-31){\line(1,1){15}}
\put(269,-29){\line(1,1){15}}
\put(256,-16){\line(1,1){30}}
\put(254,-14){\line(1,1){30}}
\put(256,14){\line(1,1){15}}
\put(254,16){\line(1,1){15}}

\multiput(255,-15)(0,2){15}{\circle*{1}}
\multiput(270,-30)(0,2){30}{\circle*{1}}
\multiput(285,-15)(0,2){15}{\circle*{1}}

\put(290,-17){$t_1-t_3$}
\put(290,13){$0$}
\put(245,-17){$0$}
\put(227,20){\small $t_1-t_3$}
\put(275,30){$t_1-t_2$}
\put(280,-32){$t_1-t_2$}

\multiput(270,-30)(-3,3){5}{\circle*{2}}
\multiput(255,15)(3,3){5}{\circle*{2}}
\multiput(285, -15)(0,4){8}{\circle*{2}}

\put(380,-30){\circle*{5}}
\put(365,-15){\circle*{5}}
\put(395,-15){\circle*{5}}
\put(365,15){\circle*{5}}
\put(395,15){\circle*{5}}
\put(380,30){\circle*{5}}

\put(380,-30){\line(-1,1){15}}
\put(395,-15){\line(-1,1){30}}
\put(395,15){\line(-1,1){15}}

\put(381,-31){\line(1,1){15}}
\put(379,-29){\line(1,1){15}}
\put(366,-16){\line(1,1){30}}
\put(364,-14){\line(1,1){30}}
\put(366,14){\line(1,1){15}}
\put(364,16){\line(1,1){15}}

\multiput(365,-15)(0,2){15}{\circle*{1}}
\multiput(380,-30)(0,2){30}{\circle*{1}}
\multiput(395,-15)(0,2){15}{\circle*{1}}

\put(400,-17){$1$}
\put(400,13){$1$}
\put(355,-17){$1$}
\put(355,13){$1$}
\put(385,30){$1$}
\put(390,-32){$1$}

\put(225,0){$-$}
\put(340,0){$=$}
\put(145,-40){\line(1,0){160}}
\end{picture}
\caption{Calculating $\partial_1(p_{s_1})$} \label{BGG div diff figure}
\end{figure}

\begin{proposition}
Both divided difference operators are well-defined $\C$-module homomorphisms and satisfy the braid relations.  If $w \in S_n$ and $s_i$ is a simple transposition then:
\begin{enumerate}
\item if $s_iw<w$ then $\delta_i p_w = p_{s_iw}$; otherwise $\delta_i p_w = 0$; and
\item if $ws_i<w$ then $\partial_i p_w = p_{ws_i}$; otherwise $\partial_i p_w = 0$.
\end{enumerate}
\end{proposition}

The proposition was proven for $\delta_i$ in \cite{T2} and for $\partial_i$ in \cite{KK}.  Both divided difference operators were discussed geometrically in unpublished lecture notes of D.\ Peterson, without complete proofs \cite{Pe}.

Comparing with Equation \eqref{second div diff prop} suggests that $\partial_i$ is the BGG divided difference operator.  This is in fact true.  The equivariant cohomology of the flag variety can also be written as the quotient $\mathbb{Z}[x_1, \ldots, x_n; y_1, \ldots, y_n]/J$ for a particular ideal $J$.  Write $R$ for the GKM description of $H^*_T(GL_n(\C)/B)$.  There is an isomorphism 
\[\varphi: \mathbb{Z}[x_1, \ldots, x_n; y_1, \ldots, y_n]/J \rightarrow R.\] 
Arabia proved in \cite{A} that $\varphi$ sends the BGG divided difference operator to the Kostant-Kumar divided difference operator, in the sense that $\varphi \circ \partial_i^{BGG} = \partial_i^{KK} \circ \varphi$. 

The second divided difference operator $\delta_i$ comes from what is known as the ``divided difference in the $y$ variables".  The map from equivariant to ordinary cohomology is the projection $\mathbb{Z}[x_1, \ldots, x_n; y_1, \ldots, y_n]/J \rightarrow \mathbb{Z}[x_1, \ldots, x_n]/I$ that sends each $y_i$ to zero.  Consequently, the divided difference operators in the $y$ variables vanish on ordinary cohomology.  However, the $\delta_i$ are well-defined $\C$-module homomorphisms on the ordinary cohomology in the GKM description.

In fact $\partial_i$ is an $\C[t_1, \ldots, t_n]$-module homomorphism though $\delta_i$ is not.  This is analogous to the statement that the star action of $w \in S_n$ is a $\C[t_1, \ldots, t_n]$-algebra homomorphism while the dot action of $w$ is a twisted $\C[t_1, \ldots, t_n]$-algebra homomorphism.  Note that $\delta_i$ and $\partial_j$ commute for each $i,j$.

The proposition also implies that all Schubert classes can be generated by a string of divided difference operators applied to the Schubert class $p_{w_0}$ of the longest permutation $w_0 \in S_n$.

We conclude this section by noting an intrinsic method of computing the BGG Schubert polynomial due to C.\ Cadman.

\begin{theorem}
(Cadman) The map $\xi: H^*_T(GL_n(\C)/B) \rightarrow \mathbb{Z}[t_1, \ldots, t_n]/I$ defined by $\xi(p) = \frac{1}{n!} \sum_{v \in S_n} p(v)$ is a $\C$-module homomorphism that
\begin{itemize}
\item sends the equivariant Schubert class $p_w$ to the BGG Schubert polynomial $\mathfrak{S}_{w^{-1}}$ for each $w \in S_n$,
\item restricts to an isomorphism on $H^*(GL_n(\C)/B)$,
\item commutes with the dot action, in the sense that $\xi (w \cdot p) = w \xi(p)$, and
\item commutes with divided difference operators, in the sense that $\xi \circ \delta_i = \partial_i \circ \xi$.
\end{itemize}
\end{theorem}
A proof can be found in \cite{T2}.

\subsection{Formulas for localization of Schubert classes}

Each divided difference operator gives a formula identifying some localizations of the Schubert classes.

\begin{proposition}\label{localization formula}
Let $p_w$ be an equivariant Schubert class.
\begin{enumerate}
\item \label{KK factorization} If $ws_i > w$ then $p_w(v) = p_w(vs_i)$ for each $v \in S_n$.
\item \label{dot factorization} If $s_iw > w$ then $s_ip_w(v) = p_w(s_iv)$ for each $v \in S_n$.
\end{enumerate}
\end{proposition}

\begin{proof}
In the first case, the hypothesis ensures that $\partial_i p_w = 0$.  Consequently we have $\frac{p_w(v) - p_w(vs_i)}{-t_{v(i)}+t_{v(i+1)}}=0$ for each $v \in S_n$, which proves the claim.

The second hypothesis ensures that $\delta_i p_w = 0$.  Consequently $\frac{p_w(v) - s_ip_w(s_iv)}{t_i-t_{i+1}}=0$ for each $v \in S_n$, which proves the claim.
\end{proof}

Figure \ref{examples of localizations} gives several examples.  
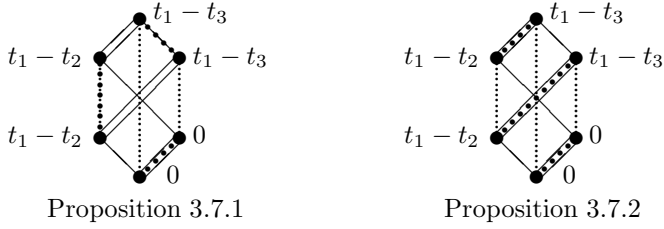
\begin{figure}[h]
\begin{picture}(360,75)(90,-45)
\put(200,-30){\circle*{5}}
\put(185,-15){\circle*{5}}
\put(215,-15){\circle*{5}}
\put(185,15){\circle*{5}}
\put(215,15){\circle*{5}}
\put(200,30){\circle*{5}}

\put(200,-30){\line(-1,1){15}}
\put(215,-15){\line(-1,1){30}}
\put(215,15){\line(-1,1){15}}

\put(201,-31){\line(1,1){15}}
\put(199,-29){\line(1,1){15}}
\put(186,-16){\line(1,1){30}}
\put(184,-14){\line(1,1){30}}
\put(186,14){\line(1,1){15}}
\put(184,16){\line(1,1){15}}

\multiput(185,-15)(0,2){15}{\circle*{1}}
\multiput(200,-30)(0,2){30}{\circle*{1}}
\multiput(215,-15)(0,2){15}{\circle*{1}}

\put(220,-17){$0$}
\put(220,13){$t_1-t_3$}
\put(150,-17){$t_1-t_2$}
\put(150,13){$t_1-t_2$}
\put(205,30){$t_1-t_3$}
\put(210,-32){$0$}

\multiput(200,-30)(3,3){5}{\circle*{2}}
\multiput(215,15)(-3,3){5}{\circle*{2}}
\multiput(185, -15)(0,4){8}{\circle*{2}}

\put(350,-30){\circle*{5}}
\put(335,-15){\circle*{5}}
\put(365,-15){\circle*{5}}
\put(335,15){\circle*{5}}
\put(365,15){\circle*{5}}
\put(350,30){\circle*{5}}

\put(350,-30){\line(-1,1){15}}
\put(365,-15){\line(-1,1){30}}
\put(365,15){\line(-1,1){15}}

\put(351,-31){\line(1,1){15}}
\put(349,-29){\line(1,1){15}}
\put(336,-16){\line(1,1){30}}
\put(334,-14){\line(1,1){30}}
\put(336,14){\line(1,1){15}}
\put(334,16){\line(1,1){15}}

\multiput(335,-15)(0,2){15}{\circle*{1}}
\multiput(350,-30)(0,2){30}{\circle*{1}}
\multiput(365,-15)(0,2){15}{\circle*{1}}

\put(370,-17){$0$}
\put(370,13){$t_1-t_3$}
\put(300,-17){$t_1-t_2$}
\put(300,13){$t_1-t_2$}
\put(355,30){$t_1-t_3$}
\put(360,-32){$0$}

\multiput(350,-30)(3,3){5}{\circle*{2}}
\multiput(335,15)(3,3){5}{\circle*{2}}
\multiput(335, -15)(3,3){10}{\circle*{2}}

\put(165,-45){Proposition 3.7.1}
\put(315,-45){Proposition 3.7.2}
\end{picture}
\caption{The localization formulas, using $s_2s_1>s_1$ and $s_1s_2>s_1$} \label{examples of localizations}
\end{figure}
Note that this proposition is most dramatic in the extreme cases.  For instance, it shows that $p_e(v) = p_e(vs_i)$ for each $v \in S_n$ and each $s_i$.  The simple transpositions $s_i=(i,i+1)$ generate all of $S_n$, so this implies $p_e(v)=p_e(u)$ for all $v,u \in S_n$.  In fact, we know $p_e(v) = 1$ for all $v \in S_n$.

\section{Applications}

\subsection{The Chevalley-Monk formula for $H^*(GL_n(\C)/B)$}

These tools give a new proof of the Chevalley-Monk formula, a formula for multiplication by a particular type of Schubert class.  We present the proof for $GL_n(\C)/B$ and then indicate how it would be generalized to arbitrary Lie type.

\begin{theorem}
(Chevalley-Monk formula) Let $p_i = p_{s_i}$ denote the Schubert class corresponding to a simple reflection in $H^*(GL_n(\C)/B)$ and choose any other Schubert class $p_w \in H^*(GL_n(\C)/B)$.  Then the product $p_ip_w$ in $H^*(GL_n(\C)/B)$ is 
\[ p_i p_w = \displaystyle \sum_{\footnotesize \begin{array}{c} (jk): \\ \ell(w(jk))=\ell(w)+1 \\  j \leq i < i+1 \leq k \end{array}}  p_{w(jk)}.\]
\end{theorem}

\begin{proof}
The dot action in the ordinary cohomology $H^*(GL_n(\C)/B$ reduces to $v \cdot p_w \cong p_w$ for each $v, w \in S_n$ since $H^*(GL_n(\C)/B) = \frac{H^*_T(GL_n(\C)/B)}{\langle t_1, \ldots, t_n \rangle H^*_T(GL_n(\C)/B)}$.  So
\[(w \cdot p_i)p_w \cong p_ip_w.\]
It suffices to compute the coefficients $(w \cdot p_i)p_w = \sum_v c_v p_v$ in $H^*(GL_n(\C)/B)$.  The classes $p_ip_w$ and $(w \cdot p_i)p_w$ are each homogenous of degree $\ell(w)+1$ so $c_v = 0$ if $\ell(v) > \ell(w)+1$.  In fact, a stronger statement holds.  If $v \not \geq w$ then
\[ \left((w \cdot p_i)p_w \right) (v) = (wp_i(w^{-1}v)) p_w(v) = 0\]
since the canonical class $p_w$ vanishes for each $v \not \geq w$.  Moreover 
\[ \left((w \cdot p_i)p_w \right) (w) = (wp_i(e)) p_w(w) = 0\]
since the canonical class $p_i$ vanishes at $e< s_i$.  Together with the degree constraint, this implies that $c_v =0$ except possibly when $v = w(jk)$ for some $(jk)$ with $\ell(w(jk))=\ell(w)+1$.  (It will become clear in a moment why we use $w(jk)$ instead of our usual convention $(j'k')w$.)

If $(jk)$ is a reflection such that $\ell(w(jk)) = \ell(w)+1$ then
\[ \left((w \cdot p_i)p_w \right) (w(jk)) = (wp_i(jk)) p_w(w(jk)).\]
The definition of canonical classes implies 
\[ \left( \sum_{\footnotesize \begin{array}{c} v: v > w, \\ \ell(v)=\ell(w)+1 \end{array}} c_v p_v \right)(w(jk)) = c_{w(jk)} p_{w(jk)}(w(jk))\]
so we conclude
\[c_{w(jk)} = \frac{(wp_i(jk)) p_w(w(jk))}{p_{w(jk)}(w(jk))}.\]

\cite{T3} proves that
\[\frac{p_w(w(jk))}{p_{w(jk)}(w(jk))} = \frac{1}{t_{w(j)} - t_{w(k)}}\] 
where $t_{w(j)}-t_{w(k)}$ is the label on the edge $w(jk) \mapsto w$.  We sketch the proof.  The only edge $w(jk) \mapsto v$ in the moment graph for which $v \geq w$ is the edge $w(jk) \mapsto w$, by comparing lengths of the permutations.  This means the localization $p_w(w(jk))$ is a constant multiple of $q$, where $q$ is the product of the labels on the edges $w(jk) \mapsto v$ such that $v \neq w$.   It is true but not obvious that the only $c \in \C$ with
\[cq - p_w(w) \in \langle t_{w(j)}-t_{w(k)} \rangle\] 
is $c=1$.  Since $p_{w(jk)}(w(jk)) = (t_{w(j)} - t_{w(k)})q$, the claim follows.

We give a new proof that $wp_i(jk) = t_{w(j)} - t_{w(k)}$, which also follows from \cite{Bi}.   If $(jk) \not > s_i$ then $p_i(jk) = 0$.  A fact about reflections in $S_n$ is that if $(jk) > s_i$ then there are $j_1, \ldots, j_e \neq i$ such that $(jk) = s_{j_e} \cdots s_{j_1}s_i s_{j_1} \cdots s_{j_e}$.  Each $s_{j_l}$ satisfies $s_is_{j_l} > s_i$ so by Proposition \ref{localization formula}.\ref{KK factorization} we have
\[p_i(s_i) = p_i(s_is_{j_1}s_{j_2} \cdots s_{j_e}).\]
Since $s_{j_l} s_i > s_i$ as well, Proposition  \ref{localization formula}.\ref{dot factorization} gives
\[p_i(s_{j_e} \cdots s_{j_1}s_is_{j_1}s_{j_2} \cdots s_{j_e}) = s_{j_e}\cdots s_{j_1}p_i(s_is_{j_1}s_{j_2} \cdots s_{j_e}).\]
The definition of canonical classes says $p_i(s_i) = t_i -t_{i+1}$, so together this shows
\[p_i(jk) =  s_{j_e}\cdots s_{j_1} (t_i-t_{i+1})\]
which is $t_j - t_k$ by definition.  Hence $wp_i(jk) = t_{w(j)} - t_{w(k)}$.
\end{proof}
%equivariant version: $w \cdot p_i = p_i - \lambda p_e$ for some weight.  so $p_i = w \cdot p_i+\lambda p_e$.  hence $p_ip_w = (w \cdot p_i)p_w + \lambda p_w$.

This result generalizes to all Lie types, but not exactly as stated: the coefficient $c_v$ may be an integer other than zero or one.  The difference is that the factorization $s_{\alpha} = s_{j_e} \cdots s_{j_1}s_i s_{j_1} \cdots s_{j_e}$ may have several indices $j_l = i$, which can increase the coefficient.  We remark that the direct generalization of this proof gives a formula that is not immediately equal to the classical Chevalley-Monk formula (see \cite[Proposition 10]{Ch}).

\subsection{Geometric representations}

The permutation group $S_n$ acts in two different ways on both the cohomology and the equivariant cohomology of the flag variety $GL_n(\C)/B$.  A priori this gives four representations: the dot representation on $H^*(GL_n(\C)/B)$ and on $H^*_T(GL_n(\C)/B)$ as well as the star representation on each.  In this section we will see that the representations on equivariant and ordinary cohomology are, in a deep algebraic sense that we make precise later, the same.  We will also show that the dot action induces copies of the trivial representation, while the star representation induces one of the most important representations: the regular representation.  In this section we use Weyl groups $W$ instead of the permutation group $S_n$ and general flag varieties $G/B$ rather than $GL_n(\C)/B$ because the result about the star representation is new.

\subsubsection{Dot representation is trivial} \label{twisted repns}

The dot action on the ordinary cohomology of the flag variety is easiest to analyze.   Consider Theorem \ref{dot action formula} in the light of the fact that the ordinary cohomology satisfies $H^*(G/B) = \frac{H^*_T(G/B)}{\langle \alpha_1, \ldots, \alpha_n \rangle H^*_T(G/B)}$. Together they say that $s_i \cdot p_v \cong p_v$ in ordinary cohomology, for every simple transposition $s_i$ and every Schubert class $p_v$.  Hence we obtain
\[w \cdot p_v = p_v\]
for each $w \in W$ and Schubert class $p_v \in H^*(G/B)$.  The trivial representation is the one-dimensional vector space $V$ upon which the (finite) group $G$ acts by $g(v) = v$ for all $g \in G$ and $v \in V$.  This identifies the representation completely.

\begin{proposition}
For each $k$, the dot action of $W$ on $H^k(G/B)$ is the direct sum of $b_k$ copies of the trivial representation, where $b_k$ is the $k^{th}$ Betti number of $G/B$.
\end{proposition}

Essentially the same result holds for the equivariant cohomology of $G/B$.  However, the dot action gives twisted $H^*_T(pt)$-module homomorphisms rather than ordinary $H^*_T(pt)$-module homomorphisms.  More formally, if $V$ is the dot representation on ordinary cohomology, then the equivariant cohomology $V \otimes H^*_T(pt)$ carries the induced action of $ \cdot \otimes \cdot$.  This twisting of the coefficient ring is not part of standard representation theory because there are few situations in which it arises naturally.   The dot action on equivariant cohomology is a rare example.  Nonetheless, much of the standard theory generalizes directly to twisted representations.  

\begin{theorem}
The dot representation on $H^*_T(G/B)$ is the direct sum of $|W|$ copies of the trivial representation, with $b_k$ copies in degree $k$ for each $k$.
\end{theorem}

The proof in \cite{T3} uses the averaging homomorphism $f: H^*_T(G/B) \rightarrow H^*_T(G/B)$ defined by $p \mapsto \frac{1}{|W|} \sum_{w \in W} w \cdot p$.  The image of the Schubert classes under $f$ forms a basis of $H^*_T(G/B)$.  (This is seen by writing $f$ as a matrix in the basis of Schubert classes, where it is strictly upper-triangular.)  Moreover, the image of each Schubert class is $W$-stable and so its $H^*_T(pt)$-span is the trivial representation.

It is natural to wonder whether group actions on equivariant cohomology can be used more generally to obtain families of twisted representations.

\subsubsection{Star representation is the regular representation}

The star representation gives a much more complicated and powerful representation: the regular representation.  The regular representation of a (finite) group $G$ is the vector space with basis elements $e_g$ for $g \in G$,  on which $h \in G$ acts by $h (e_g) = e_{hg}$.  Recall the central question in representation theory described in the Introduction: given a representation $V$, how does it decompose into irreducible representations?  In the case of the regular representation $V$, the answer is known:
\[V = \bigoplus_{\textup{irred repns } V'} (\dim V') V'\]
(see \cite[Chapter 2.4]{Se} for a classical exposition).  In practice, the regular representation is often used to deduce information about the irreducible representations.

We now prove that the star action on the equivariant cohomology of the flag variety $G/B$ (as an ungraded ring) is the regular representation of the Weyl group $W$.  This result is new so the proof is presented in full generality.

To begin, define the class $q_e = \sum_{w \in W} p_w$.  The group $W$ is finite and so the orbit of $q_e$ under the star action of $W$ is some finite set of classes in $H^*_T(G/B)$.  For each $w \in W$, define $q_w = q_e * w$.  Note that $q_v * u = q_e * (vu) = q_{vu}$ for each $u,v \in W$.

\begin{lemma}
The set $\{q_w : w \in W\}$ is linearly independent in the $H^*_T(pt)$-module $H^*_T(G/B)$.
\end{lemma}

\begin{proof}
Suppose that there is some relation $\sum_{w \in W} c_w q_w = 0$ in $H^*_T(G/B)$ for coefficients $c_w \in H^*_T(pt)$.  (The $c_w$ are not necessarily homogeneous.)

For each $v \in W$ we have
\[ \left( \sum_{w \in W} c_w q_w \right) * v = \sum_{w \in W} c_w \left( q_w * v \right) = \sum_{w \in W} c_w q_{wv} .\]
Furthermore, this sum is zero since $0 * v= 0$.  

Choose $u \in W$ so that $c_u$ has maximal degree among the $c_w$.  Let $v=u^{-1}$ in the previous calculation.  Localizing at the longest element $w_0 \in W$ gives
\[ \sum_{w\in W} c_w q_{wu^{-1}} (w_0) = c_u q_e (w_0) + \sum_{w \in W, w \neq u} c_w q_{wu^{-1}}(w_0).\]
If $w \neq u$ then 
\[q_{wu^{-1}}(w_0) = (q_e * (wu^{-1}))(w_0) = q_e(w_0uw^{-1}).\]  
Since $q_e$ is the sum of canonical classes, we obtain 
\[q_e(w_0uw^{-1}) = \sum_{s \in W} p_s(w_0uw^{-1}) = \sum_{s \leq w_0uw^{-1}} p_s(w_0uw^{-1}),\]
the last equality from the definition of canonical classes.  Since $w \neq u$ we know that $w_0uw^{-1} \neq w_0$.  It follows that $q_{wu^{-1}}(w_0)$ has degree at most $\ell(w_0)-1$.  Thus
\[c_u q_e (w_0) + \displaystyle \sum_{\footnotesize \begin{array}{c}w \in W \\ w \neq u \end{array}} c_w q_{wu^{-1}}(w_0) = (\textup{highest term in $c_u$})p_{w_0}(w_0) + \textup{ (lower degrees)}.\]
For this polynomial to be zero, its highest-degree term must be zero.  The localization $p_{w_0}(w_0)$ is nonzero by definition of canonical classes, so the highest-degree term in $c_u$ must be zero.  This means that $c_u$ is zero.  Since $c_u$ was a polynomial of maximal degree among the $c_w$, we conclude that all the coefficients $c_w$ are zero.  Hence the set $\{q_w: w \in W\}$ is linearly independent.
\end{proof}

As with the dot action, the representations induced by the star action on ordinary and equivariant cohomology are closely related.  Unlike the dot action, the star action does not twist the coefficients $H^*_T(pt)$.  Indeed, if $V$ is the representation of the star action on ordinary cohomology, then the representation on equivariant cohomology is $V \otimes H^*_T(pt)$ with the star action induced by $* \otimes 1$.

\begin{theorem}
The star representation of $W$ on $H^*_T(G/B)$ (as an ungraded module) is the regular representation with $H^*_T(pt)$-coefficients.  The star representation of $W$ on $H^*(G/B)$ (as an ungraded module) is the regular representation with $\C$-coefficients.
\end{theorem}

\begin{proof}
The previous lemma showed that $\{q_w: w \in W\}$ is linearly independent in $H^*_T(G/B)$.  The canonical classes are a basis for $H^*_T(G/B)$ as a free module over $H^*_T(pt)$.  There are $|W|$ canonical classes so $\{q_w: w \in W\}$ is also a basis for $H^*_T(G/B)$.  By construction, the action $q_w * u= q_{wu}$ gives the regular representation.  

The classes $\{q_w: w \in W\}$ surject onto the ordinary cohomology $H^*(G/B) = \frac{H^*_T(G/B)}{\langle \alpha_1, \ldots, \alpha_n \rangle H^*_T(G/B)}$ and so their images are also linearly independent.  Hence the star action on $H^*(G/B)$ is also the regular representation.
\end{proof}

\section{Questions}

We close with some open questions.

\subsection{Schubert calculus}
Schubert calculus is the study of structure constants with respect to the basis of Schubert classes in the cohomology of the flag variety, or more generally any cohomology theory of any generalized $G/P$.

For instance, Pieri formulas are a class of multiplication formulas that generalize the Chevalley-Monk formula.  S.\ Robinson proved in \cite{R} an equivariant Pieri formula for the flag variety $GL_n(\C)/B$  using methods like those outlined here.

\begin{question}
What is an equivariant Pieri formula for flag varieties $G/B$?  
\end{question}

One should not expect this generalization to be immediate or even necessarily tidy.  For instance, in \cite{PR1} and \cite{PR2}, P.\ Pragacz and J.\ Ratajski computed Pieri formulas in the (ordinary) cohomology of Grassmannians of Lie types other than $A_n$, the natural generalizations of the Grassmannian of $k$-planes in $\C^n$.  These formulas are quite combinatorially involved.  Many of the results sketched in this survey have been extended to general Grassmannians $G/P$ in \cite{T2}.  We also ask:

\begin{question}
What is an equivariant Pieri formula for Grassmannians in Lie types other than $A_n$?
\end{question}

Ideally, equivariant Pieri formulas will specialize to Pragacz-Ratajski's.

\subsection{Hessenberg varieties}

Regular semisimple Hessenberg varieties are a family of smooth, compact subvarieties of the flag variety.  We refer to them as Hessenberg varieties in this section.  Hessenberg varieties are GKM spaces that carry the dot action, though not the star action.  There are many unanswered questions about the representations on the cohomology of Hessenberg varieties.

\begin{definition} Let $X$ be a diagonal matrix with distinct eigenvalues and let $h: \{1,2,\ldots,n\} \rightarrow \{1,2,\ldots,n\}$ be a nondecreasing function that satisfies $h(i) \geq i$ for all $i$.  The (regular semisimple) Hessenberg variety $\mathcal{X}_h$ is defined as
\[\mathcal{X}_h = \{\textup{Flags }V_{\bullet}: XV_i \subseteq V_{h(i)} \textup{ for all }i=1,\ldots,n\}.\]
\end{definition}

The Hessenberg variety can be written algebraically as the collection of flags $[gB/B] \in GL_n(\C)/B$ such that the matrix $g^{-1}Xg$ vanishes in certain positions determined by $h$.  The algebraic definition permits the definition of Hessenberg varieties to be generalized to arbitrary Lie type.  See \cite{dMPS} and \cite{T4} for more.

The following gives key properties of Hessenberg varieties, omitting proofs.

\begin{proposition}
\begin{enumerate}
\item Hessenberg varieties are smooth, compact, GKM subvarieties of $GL_n(\C)/B$ for the torus $T$ \cite{dMPS}.
\item The moment graph for the Hessenberg variety $\mathcal{X}_h$ has fixed points $\{[w]: w \in S_n\}$ and has an edge $w \mapsto (ij)w$ if and only if $w^{-1}(i) \leq h(w^{-1}(j))$.
\item The dot action of $S_n$ is a well-defined group action on $H^*_T(\mathcal{X}_h)$.
\end{enumerate}
\end{proposition}

Figure \ref{all Hess} gives all Hessenberg varieties up to homeomorphism when $n=3$, with the Hessenberg function written as a sequence $h=h(1)h(2)h(3)$.
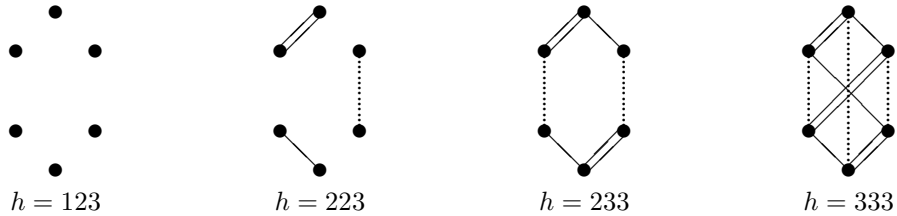
\begin{figure}[h]
\begin{picture}(360,70)(0,-40)
\put(30,-30){\circle*{5}}
\put(15,-15){\circle*{5}}
\put(45,-15){\circle*{5}}
\put(15,15){\circle*{5}}
\put(45,15){\circle*{5}}
\put(30,30){\circle*{5}}

\put(130,-30){\circle*{5}}
\put(115,-15){\circle*{5}}
\put(145,-15){\circle*{5}}
\put(115,15){\circle*{5}}
\put(145,15){\circle*{5}}
\put(130,30){\circle*{5}}

\put(130,-30){\line(-1,1){15}}

\put(116,14){\line(1,1){15}}
\put(114,16){\line(1,1){15}}

\multiput(145,-15)(0,2){15}{\circle*{1}}

\put(230,-30){\circle*{5}}
\put(215,-15){\circle*{5}}
\put(245,-15){\circle*{5}}
\put(215,15){\circle*{5}}
\put(245,15){\circle*{5}}
\put(230,30){\circle*{5}}

\put(230,-30){\line(-1,1){15}}
\put(245,15){\line(-1,1){15}}

\put(231,-31){\line(1,1){15}}
\put(229,-29){\line(1,1){15}}
\put(216,14){\line(1,1){15}}
\put(214,16){\line(1,1){15}}

\multiput(215,-15)(0,2){15}{\circle*{1}}
\multiput(245,-15)(0,2){15}{\circle*{1}}

\put(330,-30){\circle*{5}}
\put(315,-15){\circle*{5}}
\put(345,-15){\circle*{5}}
\put(315,15){\circle*{5}}
\put(345,15){\circle*{5}}
\put(330,30){\circle*{5}}

\put(330,-30){\line(-1,1){15}}
\put(345,-15){\line(-1,1){30}}
\put(345,15){\line(-1,1){15}}

\put(331,-31){\line(1,1){15}}
\put(329,-29){\line(1,1){15}}
\put(316,-16){\line(1,1){30}}
\put(314,-14){\line(1,1){30}}
\put(316,14){\line(1,1){15}}
\put(314,16){\line(1,1){15}}

\multiput(315,-15)(0,2){15}{\circle*{1}}
\multiput(330,-30)(0,2){30}{\circle*{1}}
\multiput(345,-15)(0,2){15}{\circle*{1}}

\put(13,-45){$h=123$}
\put(113,-45){$h=223$}
\put(213,-45){$h=233$}
\put(313,-45){$h=333$}
\end{picture}
\caption{All regular semisimple Hessenberg varieties when $n=3$} \label{all Hess}
\end{figure}
The graph with all exterior edges and no interior edges is the moment graph of the toric variety associated to the decomposition into Weyl chambers.  C.\ Procesi and J.\ Stembridge separately analyzed the dot action on the ordinary cohomology of this toric variety in \cite{P} and \cite{S}, giving different characterizations of the representation.  

The fundamental question is:

\begin{question}
Given $h$, what is the $S_n$-representation on $H^*(\mathcal{X}_h)$ or $H^*_T(\mathcal{X}_h)$?
\end{question}

One approach might be the following.

\begin{question}
Suppose that $h$ is a Hessenberg function with $h(i) < h(i+1)$ for some $i$.  Let $h'$ be the Hessenberg function defined by $h'(j)=h(j)$ unless $j=i$, in which case $h'(i) = h(i)+1$.  How is the $S_n$-representation on $H^*(\mathfrak{X}_h)$ related to the representation on $H^*(\mathfrak{X}_{h'})$?
\end{question}

Another question seeks to list all irreducible representations of $S_n$.

\begin{question}
Is there an (ordered) family of Hessenberg functions $h_1, \ldots, h_k$ and an ordering of the irreducible representations $\lambda_1, \ldots, \lambda_k$ of $S_n$ so that for each $i$ the irreducible representation $\lambda_i$ appears in the decomposition into irreducibles of $H^*(\mathfrak{X}_{h_i})$ but not of $H^*(\mathfrak{X}_{h_j})$ for any $j<i$?
\end{question}

\subsection{Geometric and combinatorial representations}

Other crucial geometric representations include the so-called Springer representations and representations on affine Grassmannians (infinite-dimensional analogues of flag varieties) that are important in number theory.

\begin{question}
Can more geometric representations be realized using GKM?
\end{question}

A parallel question uses the essentially combinatorial nature of GKM theory.

\begin{question}
Can other combinatorial representations be constructed geometrically using GKM theory?
\end{question}

A question that arose in Section \ref{twisted repns} was:

\begin{question}
Can equivariant cohomology be used to construct natural families of twisted group representations?
\end{question}

\end{document}